\documentclass[12pt, oneside,reqno]{amsart}
\usepackage[utf8]{inputenc}
\usepackage[svgnames,x11names,table]{xcolor}
\usepackage{graphicx} 
\usepackage{amsmath}
\usepackage{amsthm}
\usepackage{amsfonts}
\usepackage{amssymb}
\usepackage{mathtools}
\usepackage{fullpage}
\usepackage{tabularx}
\usepackage{multirow}
\numberwithin{equation}{section}
\usepackage{comment}
\usepackage{enumerate}
\usepackage{enumitem}
\usepackage{lipsum}
\usepackage{appendix}
\usepackage{tikz}
\usetikzlibrary{positioning}
\usepackage[lmargin=1in,rmargin=1in,tmargin=1in,bmargin=1in]{geometry}
\usepackage[ps,all,arc,rotate]{xy}
\usepackage{graphicx, float, epstopdf}
\usepackage{hyperref}
\usepackage{nameref}
\usepackage{centernot}
\usepackage{fancyhdr}
\usepackage{color}
\usepackage{graphics, setspace}
\usepackage{braket}
\usepackage{tikz}
\usetikzlibrary{decorations.pathmorphing, patterns,shapes}
\usepackage{bbm}
\usepackage{array,esint}

\usepackage{biblatex}
\newtheorem{theorem}{Theorem}[section]
\newtheorem*{thm}{Theorem}
\newtheorem{lemma}[theorem]{Lemma}

\newtheorem{corollary}[theorem]{Corollary}
\newtheorem{remark}[theorem]{Remark}

\title{Exponential Sums with Additive Coefficients and its Consequences to Weighted Partitions}
\author{Madhuparna Das}
\address{Department of Mathematics, University of Exeter, Exeter, EX4 4QE, United Kingdom}
\email{md679@exeter.ac.uk}
\date{}


\begin{document}

\begin{abstract}
In this article, we consider the weighted partition function $p_f(n)$ given by the generating series $\sum_{n=1}^{\infty} p_f(n)z^n = \prod_{n\in\mathbb{N}^{*}}(1-z^n)^{-f(n)}$, where we restrict the class of weight functions to strongly additive functions. Originally proposed in a paper by Yang, this problem was further examined by Debruyne and Tenenbaum for weight functions taking positive integer values. We establish an asymptotic formula for this generating series in a broader context, which notably can be used for the class of multiplicative functions. Moreover, we employ a classical result by Montgomery-Vaughan to estimate exponential sums with additive coefficients, supported on minor arcs.
\end{abstract}

\maketitle

\section{Introduction}\label{intro}

An arithmetic function $f:\mathbb{N}\to\mathbb{C}$ is called additive if it satisfies the condition $f(mn)=f(m)+f(n)$ whenever $(m,n)=1$. It is referred to as strongly additive if $f(p^k)=f(p)$ for all primes $p$ and $k\in\mathbb{N}^{*}$.\footnotetext{Throughout this paper $\mathbb{N}^{*}=\mathbb{N}\backslash\{0\}$.}
The distribution of additive functions has a deeper connection with probabilistic number theory and random walks. One of the most celebrated results in this framework is Erd\H{o}s-Kac theorem~\cite{erdoskac}, which studies the probabilistic behaviour of additive functions (see~\cite[Theorem 12.2]{elliottpart2}). 

\begin{thm}[Erd\H{o}s-Kac] Let $f$ be a real valued strongly additive function with $|f(p)|\leq 1$ for all primes $p$. Define 
\begin{align}\label{partialsums}
A_f(n)
=\sum_{p\leq n} \frac{f(p)}{p},\quad \text{and}\quad
B_f(n)
=\left(\sum_{p\leq n}\frac{f^2(p)}{p}\right)^{1/2}.
\end{align}

Assume that $B_f(n)$ is unbounded and for each fixed $\varepsilon>0$
\begin{align}\label{primedominance}
\limsup_{n\to\infty}\frac{1}{B^2_f(n)}\sum_{\substack{p\leq n\\|f(p)|>\varepsilon B_f(n)}}\frac{f^2(p)}{p}=0.    
\end{align}
Then, for any fixed $a<b$, we have that
\begin{align*}
\lim_{X\to\infty}\#\left\{n\leq X: a\leq \frac{f(n)-A_f(n)}{B_f(n)}\leq b\right\}
=\frac{1}{\sqrt{2\pi}}\int_{a}^{b}e^{-\frac{t^2}{2}}dt.
\end{align*}
\end{thm}

\medskip

The Weyl sums of the type
\begin{align}\label{expsumadditivedefinition}
S_f(N,\Theta) \coloneqq\sum_{n\leq N}f(n)e(\Theta n),
\end{align}
where $e(x)=\exp(2\pi ix)$ for $x\in\mathbb{R}$ and $f$ is an arithmetic function have a rich history in literature and are intimately related to the distribution of the function $f$ modulo 1. The global behavior of real valued additive functions is well understood, especially in the context of their statistical distributions. In~\cite[\S8]{elliottpart2}, Elliot investigated the distribution of real-valued additive functions $f$  modulo 1 type problems. For such functions, it is possible to establish necessary and sufficient conditions under which the distribution functions
\begin{align*}
F_N(z)\coloneqq\frac{1}{N}|\{n\leq N: f(n)\leq z\}|    
\end{align*}
converges to a distribution function $F$ as $N\to\infty$. This result is encapsulated in the Erd\H{o}s-Wintner theorem~\cite{erdoswintner} which asserts that $f$ possesses a limiting distribution if and only if, the following series converge
\begin{align*}
\sum_{\substack{p\in\mathbb{P}\\|f(p)|>1}}\frac{1}{p},\quad
\sum_{\substack{p\in\mathbb{P}\\|f(p)|\leq 1}}\frac{f(p)}{p},\quad
\sum_{\substack{p\in\mathbb{P}\\|f(p)|\leq 1}}\frac{f(p)^2}{p}.
\end{align*}

While the Erd\H{o}s-Wintner theorem provides qualitative conditions for the existence of a limiting distribution, exponential sums of the type~\eqref{expsumadditivedefinition} offer quantitative measures of distribution. However, the studies of additive functions are closely connected to the multiplicative functions i.e. functions $g:\mathbb{N}\to\mathbb{C}$ is an arithmetic function satisfying $g(mn)=g(m)g(n)$ whenever $(m,n)=1$. In fact, the Weyl sums $S_f(N,\Theta)$  with multiplicative coefficients have garnered significant interest due to their diverse applications. We explore these types of sums further in details culminating in an application where we establish an asymptotic formula for the weighted partition with additive functions as the weight.


\medskip

Throughout this paper $f\sim g$ means $\frac{f(x)}{g(x)}\to 1$ as $x\to\infty$. Further, $f=o(g)$ and $f=O(g)$ denotes $\left|{f(x)}/{g(x)}\right|\to 0$ and $|f(x)|\leq C|g(x)|$ as $x\to\infty$, respectively for a suitable constant $C>0$. Moreover, $\varphi(q)$ denotes the Euler totient function, $\mu(n)$ denotes the M\"{o}bius function, and $\gamma$ denotes the Euler-Mascheroni constant.

\subsection{Weighted Partitions}\label{setup}

The asymptotic behaviour of the partition function $p(n)$ has been studied intensively since the early 20th century, starting with the development in 1918 due to Hardy-Ramanujan~\cite{hardyramanujan}. It states that 
\begin{align*}
p(n)\sim \frac{1}{4\sqrt{3}n}e^{\pi\sqrt{2n/3}}    
\end{align*}
as $n\to\infty$. Let $\mathbb{U}$ denote the unit disc. The generating function of weighted partition is defined as
\begin{align}\label{generatingfunction}
\Psi_f(n)
\coloneqq \sum_{n\geq 0} p_f(n)z^n
=\prod_{n\in\mathbb{N}^{*}} (1-z^n)^{-f(n)} \quad (z\in\mathbb{U}).
\end{align}

In this paper, we delve into the behavior of~\eqref{generatingfunction} for strongly additive functions. One can modify the argument presented in this article to achieve an asymptotic result for the weighted partition of completely additive functions.

\medskip

Let $A$ be a subset of $\mathbb{N}^{*}$. The $A$-restricted integer partition, denoted as $p_A(n)$, represents the number of ways to express $n$ using elements from the set $A$. It is noteworthy that if we define $f(n)={1}_{n\in A}$, then the generating function of $p_A(n)$ can be expressed in terms of~\eqref{generatingfunction}. Consequently, we can assert that all restricted partitions can be represented in terms of weighted partition. Now, a natural question arises: Can we choose any arithmetic function in~\eqref{generatingfunction}? The answer, however, is not straightforward. While the asymptotic behaviour of restricted partitions has been extensively studied by various mathematicians\footnote{Previous studies on this topic has been explored in~\cite[\S1]{semiprime}.} for specific sets $A\subset\mathbb{N}^{*}$, tackling the problem in a more general setting with arbitrary arithmetic functions is considerably more challenging. The most general setting in this context was studied by Debruyne and Tenenbaum in~\cite{debten} by considering the weight function in~\eqref{generatingfunction} taking positive integer values. This problem initially appeared in an article by Yang~\cite{yang}.

\medskip

The pivotal step in their proof involves saddle-point solution. Notably, while the saddle-point method has been previously applied to study the asymptotic behaviour of partitions, Debruyne and Tenenbaum's work was the first to utilize it in a more general setting. Although we emphasize the saddle-point method to attain our main theorem, since we are considering a different and somewhat more general class than~\cite{debten}, handling the $L$-functions becomes challenging for our purpose. For contributions away from the saddle point, we employ sieve method in our proof. 

\medskip

It is noteworthy that Vaughan studied the asymptotic behaviour for the restricted partition function $p_{\mathbb{P}}(n)$, where $\mathbb{P}$ denotes the set of primes. However, the result obtained by Debruyne and Tenenbaum cannot be directly applied to derive Vaughan's result due to the inherent complexity associated with problems involving prime numbers. To obtain our main result this difficulty persists in more general contexts, as evidenced by the fundamental theorem of arithmetic, which asserts that the values of an additive function $f$ are uniquely determined at primes.

\medskip

We now define the collection $\mathcal{A}$ of strongly additive functions $f:\mathbb{N}\to\mathbb{R}^{+}$ that satisfy the following conditions.

\begin{itemize}
\item [(C.1)]\label{primebound} $f(p)=O(1)$ for any prime $p$.

\medskip

\item [(C.2)] $f$ satisfies the Erd\H{o}s-Wintner's condition.

\medskip

\item [(C.3)] $f$ is well-distributed in the sense that it satisfies $A$-Siegel-Walfisz criterion\footnote{For further reference see the achievement of Bombieri, Iwaniec and Friedlander in~\cite[p. 205-210]{bif}.}, i.e.,

\begin{align}\label{bombierietal}
\left|\sum_{\substack{n\leq N\\n\equiv a(\bmod{q})}}f(n)-\frac{1}{\varphi(q)}\sum_{\substack{n\leq N\\(n,q)=1}} f(n)\right|
\ll_A\frac{\sqrt{N}\|f\|}{\varphi(q)(\log N)^{A}}, \quad (a,q)=1
\end{align}
holds for any fixed $A>0$, and $\|f\|$ denotes the $\ell^2$ norm.
\end{itemize}

\begin{theorem}\label{maintheorem}
Let $f\in\mathcal{A}$. Then as $n\to\infty$,
\begin{align*}
p_f(n)
\sim c_1 n^{-\frac{3}{4}}\left(\log\log n+\psi_f\right)^{\frac{1}{4}}\exp\left(c_2(n(\log\log n+\psi_f))^{\frac{1}{2}}(1+o(1))\right),
\end{align*}
where 
\begin{align}\label{constants}
& c_1=\frac{(\zeta(2)c_f)^{\frac{1}{4}}}{\sqrt{4\pi}},\quad c_2=(\zeta(2)c_f)^{\frac{1}{2}},\nonumber\\
& \psi_f\coloneqq\gamma+\frac{\mathcal{D}_f(1)}{c_f},\quad\quad \mathcal{D}_f(1)\coloneqq\sum_{k=2}^{\infty}\frac{1}{k}\sum_{p\in\mathbb{P}}\frac{f(p)}{p^k},
\end{align}
and $c_f$ is a constant depends on $f$.
\end{theorem}

Although we utilize the saddle-point method to achieve our desired result, it is noteworthy that this approach can be viewed as a coarse version of the Hardy-Littlewood circle method. Hence, we need to establish an upper bound for~\eqref{expsumadditivedefinition} over the minor arcs.

\subsection{Weyl sums with additive coefficients}

Let $g:\mathbb{N}\to\mathbb{C}$ be a multiplicative function satisfying $|f(p)|\leq A$ for any prime $p$ and constant $A\geq 1$, and suppose that $\sum_{n\leq N}|f(n)|^2\leq A^2N$. For this class of multiplicative functions, Weyl sums have been studied over the decades, starting with the work of Daboussi~\cite{daboussi}. He proved that if $|\Theta-a/q|\leq 1/q^2$ where $(a,q)=1$ and $3\leq q\leq (N/\log N)^{\frac{1}{2}}$, then
\begin{align*}
S_g(N,\Theta)\ll_{A}\frac{N}{(\log\log q)^{\frac{1}{2}}}.
\end{align*}

This result was refined by Montgomery and Vaughan~\cite{montgomeryvaughan}, who proved that if $|\Theta-a/q|\leq 1/q^2$ where $(a,q)=1$ and $2\leq R\leq q\leq N/R$, then
\begin{align*}
S_g(N,\Theta)\ll_{A}\frac{N}{\log N}+\frac{N(\log R)^{\frac{3}{2}}}{R^{\frac{1}{2}}}.     
\end{align*}

The optimal dependence on $R$ remains an open problem and has been studied in numerous works (see~\cite{bretgran,kerr}). The work of Montgomery-Vaughan~\cite{montgomeryvaughan} is supported on the minor arcs, and as a result, it has several applications, including circle method type problems. We revisit their technique to derive the following theorem, which plays a crucial role in the proof of Theorem~\ref{maintheorem}.

\begin{theorem}\label{expsumwithadditivecoefficientforr}
Let $f:\mathbb{N}\to\mathbb{C}$ be a complex valued strongly additive function, and $ C\geq 1$ is an arbitrary constant satisfying

\begin{enumerate}
\item\label{primecondition}\hfill$\begin{aligned}[t]
|f(p)|\leq C\; \text{for each prime}\; p
\end{aligned}$\hfill\mbox{}

\medskip

\item\label{meanvaluebound}\hfill$\begin{aligned}[t]
\sum_{n\leq N}|f(n)|\ll N\log\log N
\end{aligned}$\hfill\mbox{}

\item\label{secondmoment}\hfill$\begin{aligned}[t]
\sum_{n\leq N}|f(n)|^2\ll N(\log\log N)^2.  
\end{aligned}$\hfill\mbox{}
\end{enumerate}

Suppose that
\begin{align}\label{dirichletconditionforalltheta}
\left|\Theta-\frac{a}{q}\right|\leq \frac{1}{q^2},\quad\text{where}\quad (a,q)=1
\end{align}

and $2\leq R\leq q\leq N/R$. Then we have
\begin{align*}
S_f(N,\Theta)
\ll_{C} \frac{N\log\log N}{\log N} +\frac{N\log\log N(\log R)^{\frac{3}{2}}}{R^{\frac{1}{2}}}.
\end{align*}
\end{theorem}

Note that the condition in~\eqref{meanvaluebound} can be derived from~\eqref{primecondition},~\eqref{partialsums} and Mertens's theorem (see \cite[Theorem 1.2]{harman}). Furthermore, employing the Tur\'{a}n-Kubilius inequality (see \cite[Theorem 3.1]{tenenbaum} and \cite{elliott}), one can deduce the second moment of $f$ as indicated in \eqref{secondmoment}. From Theorem~\ref{expsumwithadditivecoefficientforr}, we deduce the following.

\begin{corollary}\label{expsumwithadditivecoefficient}
For almost all $\Theta$ including all real irrational algebraic $\Theta$, we have 
\begin{align*}
S_f(N,\Theta) 
\ll_{C} \frac{N\log\log N}{\log N}\quad (N>N_0(\Theta)).  
\end{align*}
\end{corollary}

In a broader context, one can substitute the exponential function to derive an upper bound for twisted sums of the form
\begin{align*}
S_{f,\chi}(N,\Theta)\coloneqq\sum_{n\leq N}f(n)\chi(n)    
\end{align*}
where $\chi$ is a non principal character. It is worth noting that achieving a result as sharp as Corollary~\ref{expsumwithadditivecoefficient} is analogous to satisfying the 1-Siegel Walfisz criterion, as described in~\eqref{bombierietal}. Nevertheless, achieving this  particularly for large moduli would pose a challenging problem.

\subsection{Setup of the arcs}\label{arcssetup}

From~\eqref{generatingfunction} we express
\begin{align}\label{loggeneratingfunction}
\Phi_f(z)
=\sum_{k=1}^{\infty}\sum_{n=1}^{\infty}\frac{f(n)}{k}z^{nk}\quad(z\in\mathbb{U}),
\end{align}

where 
\begin{align*}
\Psi_f(z)=\exp(\Phi_f(z)).    
\end{align*}

Observe that $\Psi_f(z)$ and $\Phi_f(z)$ are analytic for $z\in\mathbb{U}$. Therefore, utilizing Cauchy's theorem we arrive at
\begin{align}\label{cauchyformula}
p_f(n)
=\rho^{-n}\int_{-\frac{1}{2}}^{{\frac{1}{2}}}\Psi_f(\rho e(\Theta))e(-n\Theta) d\Theta
=\rho^{-n}\int_{-\frac{1}{2}}^{{\frac{1}{2}}} \exp(\Phi_f(\rho e(\Theta))e(-n\Theta) d\Theta,
\end{align}
for $0<\rho<1$. For any real number $A>A_0$, we set $Q=X(\log X)^{-A}$. Additionally, for $a\in\mathbb{Z}$ and $q\in\mathbb{N}$ such that $(a,q)=1$, we define major arcs as
\begin{align*}
\mathfrak{M}(q,a)
=\bigcup_{\substack{q\leq X/Q\ (a,q)=1}}\left(\frac{a}{q}-\frac{1}{qQ},\frac{a}{q}+\frac{1}{qQ}\right),
\end{align*}
and define the minor arcs as $\mathfrak{m}=[-1/2,1/2)\backslash\mathfrak{M}$.


\section{Major arcs Analysis}\label{majorarcs}

\subsection{The fundamental estimate}

In this section, we compute the main term contribution to determine the saddle-point solution and derive the main result. We begin by studying some prerequisites on the analytic behaviour of the Dirichlet series twisted by additive coefficients.

\medskip

Let $f\in\mathcal{A}$. Define the Dirichlet series 
\begin{align}\label{lfunctionwithf}
L_f(s)=\sum_{n=1}^{\infty}\frac{f(n)}{n^s}
=\zeta(s)L_{f,\mathbb{P}}(s)
\end{align}
for $\mathrm{Re}(s)>\min(1,\sigma_f)$, where $L_{f,\mathbb{P}}(s)$ denotes the twisted series over primes, defined by
\begin{align}\label{lseriesatprimes}
L_{f,\mathbb{P}}(s)
=\sum_{p\in\mathbb{P}}\frac{f(p)}{p^s}.
\end{align}

One may consider the Dirichlet series with the multiplicative coefficient
\begin{align*}
F(u,s)
\coloneqq\sum_{n=1}^{\infty}\frac{u^{f(n)}}{n^s}
=\prod_{p\in\mathbb{P}}\left(1+\sum_{k=1}^{\infty}\frac{u^{f(p)}}{p^{ks}}\right),   
\end{align*}
for $\mathrm{Re}(s)>\beta_u$. To establish the relation~\eqref{lfunctionwithf} one may differentiate the above expression with respect to $u$ at $u=1$ .

\begin{lemma}\label{seriesexpansionoflseriesatprimes}
Consider $f\in\mathcal{A}$ and let $L_{f,\mathbb{P}}(s)$ be defined as in~\eqref{lseriesatprimes}. As $s \to 1^+$, we have
\begin{align}
L_{f,\mathbb{P}}(s)
\sim c_f\log\left(\frac{1}{s-1}\right)+O(s-1),
\end{align}
where $c_f$ is a constant dependent on $f$.
\end{lemma}

\begin{proof}
Let
\begin{align}\label{pfndefintion}
P_f(N)\coloneqq\sum_{p\leq N}f(p)\sim c_f\pi(N),
\end{align}
where $\pi(N)$ denotes the prime counting function and $c_f$ is a constant. Next, applying the Abel summation formula to~\eqref{lseriesatprimes}, we obtain
\begin{align*}
\sum_{p\in\mathbb{P}}\frac{f(p)}{p^s}
=s\int\frac{P_f(t)}{t^{s+1}}dt
\sim& c_fs\;\mathrm{Ei}((1-s)\log t)+O(1)
\sim c_f\log\left(\frac{1}{s-1}\right)+O(s-1),
\end{align*}
as $s\to1^{+}$, and $\mathrm{Ei}(x)$ denotes the exponential integral, thereby completing the proof of the lemma.
\end{proof}

Lemma~\ref{seriesexpansionoflseriesatprimes} clearly indicates that $L_{f,\mathbb{P}}(s)$ has a logarithmic singularity as $s\to 1^{+}$. Thus we normalize the series
\begin{align}\label{df}
L_{f,\mathbb{P}}(s)
=\sum_{p\in\mathbb{P}}\frac{f(p)}{p^s}
=\sum_{k=1}^{\infty}\frac{1}{k}\sum_{p\in\mathbb{P}}\frac{f(p)}{p^{ks}}-\sum_{k=2}^{\infty}\frac{1}{k}\sum_{p\in\mathbb{P}}\frac{f(p)}{p^{ks}}
\eqqcolon G_{f,\mathbb{P}}(s)-\mathcal{D}_f(s).
\end{align}

By rearranging the terms we observe that
\begin{align}\label{gptologzeta}
G_{f,\mathbb{P}}(s)
=\sum_{p\in\mathbb{P}}f(p)\sum_{k=1}^{\infty}\frac{1}{kp^{ks}}
=\sum_{p\in\mathbb{P}}f(p)\left(\log\left(\frac{1}{1-p^{-s}}\right)\right).
\end{align}

Given the relation~\eqref{pfndefintion}, we utilize Lemma~\ref{seriesexpansionoflseriesatprimes} alongside the series expansion of $\log\zeta(s)$ and employ Abel's summation formula. As a consequence, we establish that as  $s\to1^{+}$ the function $G_{f,\mathbb{P}}(s)\sim c_f\log\zeta(s)$. This result offers a profound insight into the relationship between the function $G_{f,\mathbb{P}}(s)$ and $\log\zeta(s)$ thereby serving as a pivotal tool in establishing the subsequent result.

\medskip

Before proceeding further, we will discuss the constant of $c_f$. Since $c_f$ depends on the limiting distribution of $f(p)$, we have imposed the condition that $f(p)$ is ``well-distributed" in the sense that~\eqref{pfndefintion} holds\footnote{The limiting distribution of $f(p)$ is pivotal for a suitable choice of $c_f$. For related results, we refer to~\cite{katai,hildebrand}.}. Although Lemma~\ref{seriesexpansionoflseriesatprimes} can be adapted with minor modifications for a broader class of additive functions, we assume the relation given in~\eqref{pfndefintion} for simplicity. For such cases of $f$, one might choose
\begin{align*}
c_f\coloneqq\lim_{N\to\infty}\frac{1}{\pi(N)}\sum_{p\leq N}f(p).    
\end{align*}

The main objective of Lemma~\ref{seriesexpansionoflseriesatprimes} is to analyze the logarithmic singularities of the series $L_{f,\mathbb{P}}(s)$ to establish the following lemma. With an appropriate choice of $c_f$, one can deduce the behavior of $P_f(N)$ and consequently the logarithmic singularities of $L_{f,\mathbb{P}}(s)$. In cases where $f(p)$ exhibits rapid fluctuations, stabilizing it can be achieved using the following series
\begin{align*}
c_f\coloneqq\limsup_{N \to \infty} \frac{1}{B_f(N) \log N} \sum_{p \leq N} \frac{f(p) \log p}{p},   
\end{align*}
from which one can study the behavior of $P_f(N)$\footnote{This result was developed by Hildebrand, see~\cite[Theorem 3.12]{tenenbaum} and works for a wider class of additive functions, including real positive valued completely additive functions.}. However, such cases can present complex challenges. Therefore, a variation of the next lemma can be formulated for different classes of real positive-valued additive functions, depending on the distribution of $f(p)$.

\medskip

The key point in our proof of Theorem~\ref{maintheorem} hinges on the behavior of $f(p)$, emphasizing the essential nature of its distribution. Therefore, to establish the main result, it is crucial to ensure that $f(p)$ is uniformly distributed.

Now we study the asymptotic behaviour of $\Phi_f(z)$ as $z\to 1$. 

\begin{lemma}\label{majorarcsestimate}
Suppose that $\rho=e^{-\frac{1}{X}}$ and $m\in\mathbb{N}$. Then as $X\to\infty$, one has that
\begin{align*}
\Phi_{f,(m)}(\rho)
=\left(\rho\frac{d}{d\rho}\right)^{m}\Phi_f(\rho)
=X^{m+1}\zeta(2)\Gamma(m+1)c_f\left(\log\log X+\psi_f+\frac{\mathcal{C}_m}{\log X}\right)\left(1+O\left(\frac{1}{\log X}\right)\right),
\end{align*}
and
\begin{align*}
\Phi^{(m)}_f(\rho)
=X^{m+1}\zeta(2)\Gamma(m+1)c_f\left(\log\log X+\psi_f+\frac{\mathcal{C}_m}{\log X}\right)\left(1+O\left(\frac{1}{\log X}\right)\right),
\end{align*}
where the constants $\psi_f$ and $c_f$ are defined in~\eqref{constants}. Here,
\begin{align*}
\mathcal{C}_m
=\frac{\Gamma'}{\Gamma}(m+1)+\gamma+\frac{\zeta'}{\zeta}(2)
\end{align*}
and $\frac{\Gamma'}{\Gamma}(x)$ denotes the polygamma function.
\end{lemma}

\begin{proof}
Replacing $z=\rho$ in~\eqref{loggeneratingfunction}, we obtain
\begin{align*}
\Phi_f(\rho)
=\sum_{k=1}^{\infty}\sum_{n=1}^{\infty}\frac{f(n)}{k}\rho^{nk}.
\end{align*}

The Cahen-Mellin formula for the exponential function is given by
\begin{align}\label{mellinforexp}
e^{-nk/X}
=\frac{1}{2\pi i}\int_{(c)}\Gamma(s)\left(\frac{nk}{X}\right)^{-s}ds.
\end{align}

Recalling the definition of $\rho$ and employing identity~\eqref{mellinforexp} and~\eqref{lfunctionwithf}, we express
\begin{align}\label{beforecontour}
\left(\rho\frac{d}{d\rho}\right)^{m}\Phi_f(\rho)
=&\sum_{k=1}^{\infty}\frac{1}{k}\sum_{n=1}^{\infty}f(n)(kn)^{m}e^{-kn/X}\nonumber\\
=&\frac{1}{2\pi i}\int_{(c)}\left(\sum_{n=1}^{\infty}\frac{f(n)}{n^{s-m}}\right)\left(\sum_{k=1}^{\infty}\frac{1}{k^{s+1-m}}\right)\Gamma(s)X^s\;ds\nonumber\\
=&\frac{1}{2\pi i}\int_{(c)}\zeta(s-m)L_{f,\mathbb{P}}(s-m)\zeta(s+1-m)\Gamma(s)X^s\;ds,
\end{align}
for $c>1$. Considering the expression~\eqref{beforecontour} and emphasizing~\eqref{df}, we set
\begin{align*}
\mathcal{I}(m,X)
\coloneqq&\frac{1}{2\pi i}\int_{(c)}\zeta(s-m)L_{f,\mathbb{P}}(s-m)\zeta(s+1-m)\Gamma(s)X^s\;ds\\
=&\frac{1}{2\pi i}\int_{(c)}\zeta(s-m)G_{f,\mathbb{P}}(s-m)\zeta(s+1-m)\Gamma(s)X^s\;ds\\
&\quad\quad\quad\quad-\frac{1}{2\pi i}\int_{(c)}\zeta(s-m)\mathcal{D}_{f}(s-m)\zeta(s+1-m)\Gamma(s)X^s\;ds\\
=&\mathcal{I}_1(m,X)+\mathcal{I}_2(m,X).
\end{align*}

\begin{figure}
\centering
\begin{tikzpicture}
\draw (3.5,3) -- (7.5,3) node[anchor=north west] {};
\draw (-2,-2) -- (-2,8) node[anchor=south east] {};
\draw[decoration = {zigzag,segment length = 1.2mm, amplitude = 0.5mm},decorate] (-2,3) -- (3.5,3);
\draw[dashed, black] (0.5,-2) -- (0.5,8);
\draw[dashed, red] (3.5,-2) -- (3.5,8);
\fill[blue] (3.5,3) circle (3 pt);
\draw (0,3) node[anchor=north west] {\tiny{$\frac{1}{2}$}};
\draw (6.8,7.7) -- (6.8,-1.7);

\draw[dashed, black] (0.5,5) parabola (-1.5,8);
\draw[dashed, black] (0.5,5) parabola (2.5,8);
\draw[dashed, black] (0.5,1) parabola (-1.5,-2);
\draw[dashed, black] (0.5,1) parabola (2.5,-2);

\draw (1.5,5.8) .. controls (1.75,6) and (2.35,7.5) .. (2.4,7.7);
\draw (1.5,0.2) .. controls (1.75,0) and (2.35,-1.5) .. (2.4,-1.7);

\draw (2.4,7.7) -- (6.8,7.7);
\draw (2.4,-1.7) -- (6.8,-1.7);
\draw (6.8,7.7) -- (6.8,-1.7);

\draw (1.5,5.8) -- (1.5,3.55);
\draw (1.5,2.4) -- (1.5,0.2);
\draw (1.5,3.55) -- (2.7,3.55);
\draw (1.5,2.4) -- (2.7,2.4);

\draw[black, dashed] (-2,2.4) -- (1.5,2.4); 
\draw[black, dashed] (-2,3.55) -- (1.5,3.55);

\path[draw,line width=0.4pt,rotate=-180] (-2.7,-2.4) arc (-325:-35:1);

\draw[thick,->] (3.5,3) -- (4.25,3.65);

\fill (0.5,3) circle (1.7 pt);
\fill (-2,3) circle (1.5 pt);
\fill[blue] (1.5,3.55) circle (2.2 pt);
\fill[blue] (2.7,3.55) circle (2.2 pt);
\fill[blue] (1.5,2.4) circle (2.2 pt);
\fill[blue] (2.7,2.4) circle (2.2 pt);
\fill[blue] (6.8,7.7) circle (2.2 pt);
\fill[blue] (6.8,-1.7) circle (2.2 pt);
\fill[blue] (6.8,3) circle (2.2 pt);
\fill[blue] (1.5,5.8) circle (2.2 pt);
\fill[blue] (2.4,7.7) circle (2.2 pt);
\fill[blue] (2.4,-1.7) circle (2.2 pt);
\fill[blue] (1.5,0.2) circle (2.2 pt);
\fill (-2,2.4) circle (1.5 pt);
\fill (-2,3.55) circle (1.5 pt);
\fill (-2,-1.7) circle (1.5 pt);
\fill (-2,7.7) circle (1.5 pt);

\draw (6.8,3) node[anchor=north west] {$c_0$};
\draw (6.8,7.7) node[anchor=north west] {$c_0+iT$};
\draw (6.8,-1.7) node[anchor=north west] {$c_0-iT$};
\draw (1.5,5) node[anchor=north west] {\tiny{$1-\frac{c}{\log T}+ih$}};
\draw (1.5,1.7) node[anchor=north west] {\tiny{$1-\frac{c}{\log T}-ih$}};
\draw (2.2,7.5) node[anchor=north west] {\tiny{$1-\frac{c}{\log T}+iT$}};
\draw (2.2,-1) node[anchor=north west] {\tiny{\tiny{$1-\frac{c}{\log T}-iT$}}};
\draw (-2.6,7.7) node[anchor=north west] {\tiny{$iT$}};
\draw (-2.75,-1.7) node[anchor=north west] {\tiny{$-iT$}};
\draw (-2.6,3.7) node[anchor=north west] {\tiny{$ih$}};
\draw (-2.75,2.5) node[anchor=north west] {\tiny{$-ih$}};

\draw (-2,3) node[anchor=north west] {\tiny{0}};
\draw (3.5,3) node[anchor=north west] {\tiny{1}};
\draw (4,3.5) node[anchor=north west] {{$r$}};

\draw (1,6.7) node[anchor=north west] {{$\Xi_3$}};
\draw (1,0) node[anchor=north west] {{$\Xi_7$}};
\draw (4.5,-1.7) node[anchor=north west] {{$\Xi_8$}};
\draw (6.8,5) node[anchor=north west] {{$\Xi_1$}};
\draw (4.5,8.3) node[anchor=north west] {{$\Xi_2$}};
\draw (4.5,3) node[anchor=north west] {{$\Xi_5$}};
\draw (1.8,2.3) node[anchor=north west] {\tiny{$\Xi_6$}};
\draw (1.8,4.1) node[anchor=north west] {\tiny{$\Xi_4$}};

\draw[->, line width=0.4mm] (6.8,4) -- (6.8,4.1);
\draw[->, line width=0.4mm] (6.8,6) -- (6.8,6.1);
\draw[->, line width=0.4mm] (6.8,-1) -- (6.8,-0.9);
\draw[->, line width=0.4mm] (6.8,1) -- (6.8,1.1);
\draw[<-, line width=0.4mm] (5.8,7.7) -- (5.9,7.7);
\draw[<-, line width=0.4mm] (3.8,7.7) -- (3.9,7.7);
\draw[->, line width=0.4mm] (5.8,-1.7) -- (5.9,-1.7);
\draw[->, line width=0.4mm] (3.8,-1.7) -- (3.9,-1.7);
\draw[->, line width=0.4mm] (1.5,4.7) -- (1.5,4.6);
\draw[->, line width=0.4mm] (1.5,4.2) -- (1.5,4.1);
\draw[->, line width=0.4mm] (1.5,1.7) -- (1.5,1.6);
\draw[->, line width=0.4mm] (1.5,1.2) -- (1.5,1.1);
\draw[->, line width=0.4mm] (1.9,3.55) -- (2,3.55);
\draw[->, line width=0.4mm] (2.3,3.55) -- (2.4,3.55);
\draw[<-, line width=0.4mm] (1.9,2.4) -- (2,2.4);
\draw[<-, line width=0.4mm] (2.3,2.4) -- (2.4,2.4);
\draw[<-, line width=0.4mm] (1.7,6.145) -- (1.75,6.19);
\draw[<-, line width=0.4mm] (2,6.7) -- (2.04,6.77);
\draw[->, line width=0.4mm] (1.72,-0.1) -- (1.77,-0.15);
\draw[->, line width=0.4mm] (2.14,-1) -- (2.18,-1.04);
\draw[->, line width=0.4mm] (1.72,-0.1) -- (1.77,-0.15);
\draw[->, line width=0.4mm] (3.4,3.96) -- (3.45,3.95);
\draw[<-, line width=0.4mm] (3.55,1.99) -- (3.6,1.98);

\end{tikzpicture}    
\caption{Contour Representation of $\Xi$}\label{contourpic}
\end{figure}
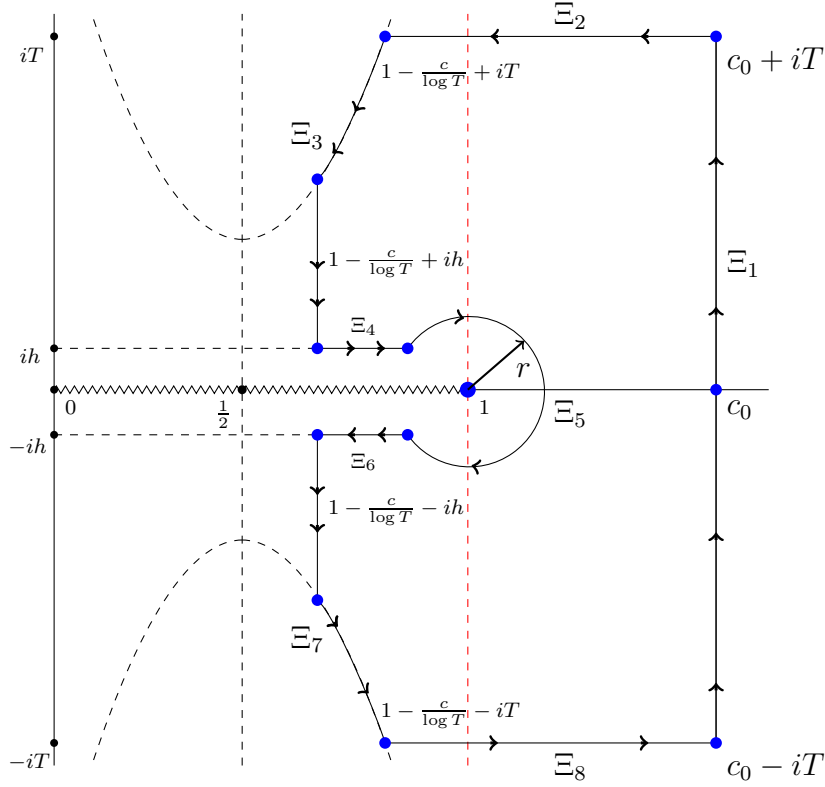

We begin with the integral $\mathcal{I}_1(m,X)$, which constitutes the main term. Considering Lemma~\ref{seriesexpansionoflseriesatprimes} and~\eqref{gptologzeta}, we note that the series $G_{f,\mathbb{P}}(s)$ exhibits behavior asymptotically equivalent to $c_f\log\zeta(s)$ and has a logarithmic singularity at $s=1$. Therefore, we employ the approach outlined in~\cite{vgn} to handle this singularity effectively.

\medskip

However, we must also account for the pole of $\zeta(s)$ in addition to the logarithmic singularity, which complicates our contour more than Lemma 2.3 of~\cite{vaughan}. As depicted in Figure~\ref{contourpic}, we set the height of the branch cut at $h>0$ and the radius of the semicircle of the Hankel contour to $r>0$.

Denoting this contour as $\Xi$ and recognizing the analyticity of the integrand within $\Xi$, we can apply Cauchy's theorem to assert that
\begin{align}\label{contourintegral}
\mathcal{I}_1(m,X)
=&\frac{1}{2\pi i}\oint_{\Xi} \zeta(s-m)G_{f,\mathbb{P}}(s-m)\zeta(s+1-m)\Gamma(s)X^s\;ds\nonumber\\
=&\frac{1}{2\pi i}\left(\oint_{\Xi_1}+\cdots+\varointclockwise_{\Xi_5}+\cdots+\oint_{\Xi_8}\right)\zeta(s-m)G_{f,\mathbb{P}}(s-m)\zeta(s+1-m)\Gamma(s)X^s\;ds
=0.
\end{align}

We start by computing the contribution arising from the semicircle of the Hankel contour
\begin{align*}
\mathcal{I}_{1,\Xi_5}(m,X)
=\frac{1}{2\pi i}\varointclockwise_{\Xi_5}\zeta(s-m)G_{f,\mathbb{P}}(s-m)\zeta(s+1-m)\Gamma(s)X^sds.
\end{align*}

Let $\eta(\theta)=1+re^{i\theta}$ with  $\theta\in[\pi,-\pi]$. Since the semicircle is clockwise we use the parametrization $s-m=1+re^{i\theta}$ as $r\to 0$. Hence, we obtain the series expansion
\begin{align}\label{seriesexpansionforzetalogzeta}
\zeta(1+re^{i\theta})G_{f,\mathbb{P}}(1+re^{i\theta})r
=c_f\zeta(1+re^{i\theta})\log\zeta(1+re^{i\theta})r+O_f(r)
=c_fe^{-i\theta}\log\frac{e^{-i\theta}}{r}+O_f(r).
\end{align}

Thus integrating over the semicircle $\eta(\theta)$ and utilizing~\eqref{seriesexpansionforzetalogzeta} yields
\begin{align}\label{i5final}
\mathcal{I}_{1,\Xi_5}(m,X)
=&\frac{c_f}{2\pi i}\int_{\pi}^{-\pi}\zeta(1+re^{i\theta})\log\zeta(1+re^{i\theta})\zeta(2+re^{i\theta})\Gamma(m+1+re^{i\theta})\nonumber\\
&\quad\quad\quad\quad\quad\quad\quad\quad\quad\quad\quad\quad\quad\quad\quad\quad\quad\quad X^{m+1+re^{i\theta}}(ire^{i\theta})d\theta+O_f(r)\nonumber\\
=&\frac{c_f\zeta(2)\Gamma(m+1)X^{m+1}}{2\pi i}\int_{\pi}^{-\pi}\zeta(1+re^{i\theta})\log\zeta(1+re^{i\theta})(ire^{i\theta})d\theta +O_f(r)\nonumber\\
=&\frac{c_f\zeta(2)\Gamma(m+1)X^{m+1}}{2\pi}\int_{\pi}^{-\pi}\log\frac{e^{-i\theta}}{r}\;d\theta+O_f(r)\nonumber\\
=&c_f\zeta(2)\Gamma(m+1)X^{m+1}\log r+O_f(r).
\end{align}

Next, we calculate the branch cut, denoted by $\mathcal{I}_{1,\Xi_4}(m,X)$ and $\mathcal{I}_{1,\Xi_6}(m,X)$, corresponding to the top and bottom branches of the logarithms, respectively as shown in Figure~\ref{contourpic}.
\begin{align*}
\mathcal{I}_{1,\Xi_4}(m,X)-\mathcal{I}_{1,\Xi_6}(m,X)
=\frac{1}{2\pi i}\left(\int_{\Xi_4}-\int_{\Xi_6}\right)\zeta(s-m)G_{f,\mathbb{P}}(s-m)\zeta(s+1-m)\Gamma(s)X^s\;ds.
\end{align*}

Set
\begin{align}\label{logzetaidentity}
\log\zeta(s)=-\log(s-1)+H(s),\quad\text{where}\quad H(s)=\log((s-1)\zeta(s)).   
\end{align}

We truncate the branch with height $h>0$, leading to a change of variable $s-m=1-(u+ih)$. Utilizing the identities~\eqref{seriesexpansionforzetalogzeta},~\eqref{logzetaidentity} and $\log(a+ib)=\frac{1}{2}\log(a^2+b^2)+i\theta$, we have
\begin{align}\label{i4}
\mathcal{I}_{1,\Xi_4}(m,X)
=&\frac{1}{2\pi i}\int_{m+1-\frac{c}{\log T}}^{m+1}\zeta(s-m)G_{f,\mathbb{P}}(s-m)\zeta(s+1-m)\Gamma(s)X^s\;ds\nonumber\\
=&\frac{c_f}{2\pi i}\int_{r}^{r+\frac{c}{\log T}} \zeta(1-(u+ih))\left(-\frac{1}{2}\log(u^2+h^2)+i\theta+H(1-(u+ih))\right)\nonumber\\
&\quad\quad\quad\quad\quad\quad \zeta(2-(u+ih))\Gamma(m+1-(u+ih))X^{m+1-(u+ih)} du +O_f(h)
\end{align}

and
\begin{align}\label{i6}
\mathcal{I}_{1,\Xi_6}(m,X)
=&\frac{c_f}{2\pi i}\int_{m+1-\frac{c}{\log T}}^{m+1}\zeta(s-m)G_{f,\mathbb{P}}(s-m)\zeta(s+1-m)\Gamma(s)X^s\;ds\nonumber\\
=&\frac{c_f}{2\pi i}\int_{r}^{r+\frac{c}{\log T}} \zeta(1-(u+ih))\left(-\frac{1}{2}\log(u^2+h^2)-i\theta+H(1-(u+ih))\right)\nonumber\\
&\quad\quad\quad\quad\quad\quad \zeta(2-(u+ih))\Gamma(m+1-(u+ih))X^{m+1-(u+ih)} du +O_f(h).    
\end{align}

Taking $\theta=\pi$ and $\theta=-\pi$ in~\eqref{i4} and~\eqref{i6} respectively yield
\begin{align*}
\mathcal{I}_{1,\Xi_{4,6}}(m,X)
=&\mathcal{I}_{1,\Xi_4}(m,X)-\mathcal{I}_{1,\Xi_6}(m,X)\nonumber\\
=&\frac{c_f}{2\pi i}\int_{r}^{r+\frac{c}{\log T}}(2\pi i)\zeta(1-(u+ih))\zeta(2-(u+ih))\nonumber\\
&\quad\quad\quad\quad\Gamma(m+1-(u+ih)) X^{m+1-(u+ih)} du +O_f(h).
\end{align*}

Letting $h\to 0$ we arrive at
\begin{align}\label{i46}
\mathcal{I}_{1,\Xi_{4,6}}(m,X)
=c_f\int_{r}^{r+\frac{c}{\log T}}\zeta(1-u)\zeta(2-u)\Gamma(m+1-u)X^{m+1-u}du.
\end{align}

We recall the following identities before proceeding with the further computation.
\begin{align*}
\Gamma(m+1-u)\zeta(2-u)\zeta(1-u)
=\zeta(2)\Gamma(m+1)\left(-\frac{1}{u}+\frac{\Gamma'}{\Gamma}(m+1)+\gamma+\frac{\zeta'}{\zeta}(2)\right)+O(u).
\end{align*}

Simplifying the above expression, we define
\begin{align*}
\mathcal{C}_m\coloneqq\frac{\Gamma'}{\Gamma}(m+1)+\gamma+\frac{\zeta'}{\zeta}(2). 
\end{align*}

Hence, the integrand \eqref{i46} can be expressed as
\begin{align}\label{i46final}
\mathcal{I}_{1,\Xi_{4,6}}(m,X)
=&c_f\zeta(2)\Gamma(m+1)X^{m+1}\int_{r}^{r+\frac{c}{\log T}}\left(-\frac{1}{u}+\mathcal{C}_m\right)X^{-u}du+O\left(X^{m+1}\int_{r}^{r+\frac{c}{\log T}}X^{-u}du\right)\nonumber\\
=&c_f\zeta(2)\Gamma(m+1)X^{m+1}\bigg(\Gamma\left(0,\left(r+{c}/{\log T}\right)\log X\right)-\Gamma(0,r\log X)\nonumber\\
&\quad\quad\quad\quad\quad\quad
+\frac{\mathcal{C}_mX^{-r}(1-X^{-\frac{c}{\log T}})}{\log X}\bigg)+O\left(\frac{X^{m+1}(1-X^{-\frac{c}{\log T}})}{X^{r}\log X}\right),
\end{align}
where $\Gamma(a,z)$ represents the incomplete gamma function. Let us observe that
\begin{align*}
\Gamma(0,r\log X)=-\gamma-\log r-\log\log X+O(r)
\end{align*}
as $r\to 0$. Subsequently, combining expressions~\eqref{i5final} and~\eqref{i46final}, yields
\begin{align*}
\mathcal{I}_{1,\Xi_{4,6}}(m,X)-\mathcal{I}_{1,\Xi_5}(m,X)
=&c_f\zeta(2)\Gamma(m+1)X^{m+1}\bigg(\Gamma\left(0,r+\frac{c}{\log T}\right)+\gamma+\log\log X\\
&+\frac{\mathcal{C}_mX^{-r}(1-X^{-\frac{c}{\log T}})}{\log X}\bigg)+O\left(\frac{X^{m+1}(1-X^{-\frac{c}{\log T}})}{X^{r}\log X}\right)+O_f(r).
\end{align*}

Since we have removed the logarithmic singularity, we can let $r \to 0$, which simplifies the above expression. Choosing $T =\exp(\sqrt{\log X})$, we have $X^{-\frac{c}{\log T}}=e^{-c\sqrt{\log X}}$ as $X\to\infty$. Therefore,
\begin{align*}
\Gamma\left(0,\frac{c\log X}{\log T}\right)
=e^{-c\sqrt{\log X}+O\left(\frac{1}{X}\right)}\left(\frac{-1+c\sqrt{\log X}}{c^2\log X}+O\left(\frac{1}{X}\right)\right).
\end{align*}

Thus the integrand gives
\begin{align}\label{i1final}
\mathcal{I}_1(m,X)
=c_f\zeta(2)\Gamma(m+1)X^{m+1}\left(\log\log X+\gamma+\frac{\mathcal{C}_m}{\log X}\right)\left(1+O\left(\frac{1}{\log X}\right)\right).
\end{align}

Considering that $\Gamma(s)$ has exponential decay for $s>1$, the contribution arising from $\{\Xi_1,\Xi_2,\Xi_3,\Xi_7,\Xi_8\}$ will have lower order terms compared to~\eqref{i1final}. As these terms are symmetric, we show the computation for one of them, and others follow similarly. Shifting the line of integration to $\mathrm{Re}(s)=c_0$ for any small $c_0>\delta_0$, we have 
\begin{align*}
\mathcal{I}_{1,\Xi_2}(m,X) 
=&\zeta(s-m)G_{f,\mathbb{P}}(s-m)\zeta(s+1-m)\Gamma(s)X^{s}ds\\
=&\frac{1}{2\pi i}\int_{c-iT}^{c+iT}\zeta(s-m)G_{f,\mathbb{P}}(s-m)\zeta(s+1-m)\Gamma(s)X^{s}ds
\ll X^{c_0}.
\end{align*}

In order to complete the proof, we finally estimate
\begin{align*}
\mathcal{I}_2(m,X)
=\frac{1}{2\pi i}\oint_{\Xi}\zeta(s-m)\mathcal{D}_f(s-m)\zeta(s+1-m)\Gamma(s)X^sds,
\end{align*}
following the preceding argument. Due to the absence of the logarithmic singularities, the subsequent computations are simpler. As defined in~\eqref{df}, we account for a simple pole of $\zeta(s-m)$. Thus, integrating over the semicircle $\eta(\theta)$ and letting $r\to 0$, the contribution from the semicircle of the Hankel contour is given by
\begin{align*}
\mathcal{I}_{2,\Xi_5}(m,X)
=&\frac{1}{2\pi i}\varointclockwise_{\Xi_5}\zeta(s-m)\mathcal{D}_f(s-m)\zeta(s+1-m)\Gamma(s)X^sds \nonumber\\
=&\frac{1}{2\pi i}\int_{\pi}^{-\pi}\zeta(1+re^{i\theta})\mathcal{D}_f(1+re^{i\theta})\zeta(2+re^{i\theta})\Gamma(m+1+re^{i\theta})X^{m+1+re^{i\theta}}(rie^{i\theta})d\theta \nonumber\\
=&\frac{\zeta(2)\mathcal{D}_f(1)\Gamma(m+1)X^{m+1}}{2\pi}\int_{\pi}^{-\pi}\zeta(1+re^{i\theta})re^{i\theta}d\theta\nonumber\\
=&\frac{\zeta(2)\mathcal{D}_f(1)\Gamma(m+1)X^{m+1}}{2\pi}\int_{\pi}^{-\pi}d\theta+O(r^2)\nonumber\\
=&\zeta(2)\mathcal{D}_f(1)\Gamma(m+1)X^{m+1}+O(r^2),
\end{align*}

where $\mathcal{D}_f(1)\coloneqq\sum_{k=2}^{\infty}\frac{1}{k}\sum_{p\in\mathbb{P}}\frac{f(p)}{p^k}$ is a constant depends on $f$. Similarly, the top brunch cut with height $h>0$ gives
\begin{align}\label{i24}
\mathcal{I}_{2,\Xi_{4}}(m,X)
=&\frac{1}{2\pi i}\int_{m+1-\frac{c}{\log T}}^{m+1}\zeta(s-m)\mathcal{D}_f(s-m)\zeta(s+1-m)\Gamma(s)X^sds \nonumber\\
=&\frac{1}{2\pi i}\int_{r}^{r+\frac{c}{\log T}}\zeta(1-(u+ih))\mathcal{D}_f(1-(u+ih))\zeta(2-(u+ih))\nonumber\\
&\quad\quad\quad\quad\quad\quad\quad\quad\quad\quad\quad\quad\quad\quad\quad\Gamma(m+1-(u+ih))X^{m+1-(u+ih)}du.
\end{align}

Likewise, the bottom cut with opposite orientation yields the same expression as in \eqref{i24}. Consequently, the term $\mathcal{I}_{2,\Xi_{4,6}}(m,X)$ vanishes. Thus 
\begin{align}\label{i2final}
\mathcal{I}_{2,\Xi_{4,6}}(m,X)-\mathcal{I}_{2,\Xi_{5}}(m,X)
=-\zeta(2)\mathcal{D}_f(1)\Gamma(m+1)X^{m+1}.
\end{align}

Lastly, using induction on $m$, we express
\begin{align*}
\rho^m\Phi_f^{(m)}(\rho)
=\sum_{i=1}^{m}c_{i,m}\rho^{i}\Phi_f^{(i)}(\rho),
\end{align*}
where $c_{i,m}$ are real-valued coefficients with $c_{m,m}=1$. By combining the contributions from the contour integrals~\eqref{i1final} and~\eqref{i2final}, we conclude the lemma.

\end{proof}

\subsection{The mean value estimate}

In this section, we determine the major arcs contribution away from the saddle-point solution by estimating its mean value. Before delving into that we study the behaviour of $f\in\mathcal{A}$ over arithmetic progression. In this step, we employ sieve method.

\medskip

Let $\chi$ be a primitive Dirichlet character modulo $q$. Then the orthogonality relation states that
\begin{align}\label{orthogonalityrelation}
\sum_{\chi(\bmod{q})} \chi(a)=
\begin{cases}
\varphi(q) & \text{ if $a\equiv 1(\bmod{q})$,}\\
0 & \text{ otherwise.}
\end{cases}
\end{align}

We utilize the $A$-Siegel-Walfisz criterion to establish the following lemma.

\begin{lemma}\label{functionprogressionlemma}
Let $f$ be a strongly additive function and $f\in\mathcal{A}$. Suppose that $q\leq (\log N)^A$ and that $(\ell,q)=1$. Then we have
\begin{align}\label{functionprogression}
\sum_{\substack{n\leq N\\n\equiv\ell(\bmod{q})}}f(n)
=\frac{N}{\varphi(q)}\log\log N+N\mathcal{C}(q)+O\left(\frac{\sqrt{N}\|f\|}{\varphi(q)(\log N)^{A}}\right),
\end{align}
where $\mathcal{C}(q)$ is a constant.
\end{lemma}

\begin{proof}
By the orthogonality relation of the Dirichlet characters (as given in~\eqref{orthogonalityrelation}) and~\eqref{bombierietal}, we have
\begin{align*}
\sum_{\substack{n\leq N\\n\equiv\ell(\bmod{q})}}f(n)
=&\frac{1}{\varphi(q)}\sum_{\chi(\bmod{q})}\chi(\ell)\sum_{n\leq N}f(n)\Bar{\chi}(n)\\
=&\frac{1}{\varphi(q)}\left(\sum_{\substack{\chi(\bmod{q})\\ \chi=\chi_0}}\chi(\ell)\sum_{n\leq N}f(n)\Bar{\chi}(n)+\sum_{\substack{\chi(\bmod{q})\\ \chi\ne\chi_0}}\chi(\ell)\sum_{n\leq N}f(n)\Bar{\chi}(n)\right)\\
=&\frac{1}{\varphi(q)}\sum_{\substack{n\leq N\\(n,q)}}f(n)+O\left(\frac{\sqrt{N}\|f\|}{\varphi(q)(\log N)^{A}}\right),
\end{align*}
for any real $A>0$. 


\medskip

Considering the main term, we conclude\footnote{Additionally, one can say that $\sum_{\substack{p\leq N\\p|q}}\frac{f(p)}{p}=O(1)$.}
\begin{align*}
\frac{1}{\varphi(q)}\sum_{\substack{n\leq N\\(n,q)}}f(n)+O\left(\frac{\sqrt{N}\|f\|}{\varphi(q)(\log N)^{A}}\right)
=&\frac{1}{\varphi(q)}\sum_{p\leq N}\left[\frac{N}{p}\right]f(p)+O\left(\frac{\sqrt{N}\|f\|}{\varphi(q)(\log N)^{A}}\right)\\
=&\frac{N}{\varphi(q)}\sum_{\substack{p\leq N\\(p,q)=1}}\frac{f(p)}{p}+\mathcal{C}(q)+O\left(\frac{\sqrt{N}\|f\|}{\varphi(q)(\log N)^{A}}\right)\\
=&\frac{N}{\varphi(q)}A_f(N)+\mathcal{C}(q)+O\left(\frac{\sqrt{N}\|f\|}{\varphi(q)(\log N)^{A}}\right),
\end{align*}
completing the proof.
\end{proof}



We proceed to establish the following result aimed at estimating the mean value. Later, this result will help us show that the major arcs contribution away from saddle-point will be subdued by the main term.

\begin{lemma}\label{meanvaluelemma}
Suppose that $B$ is a fixed positive real number and $X>X_0(B)$. Let $a\in\mathbb{Z}$ and $q\in\mathbb{N}$ such that $(a,q)=1$ and $q\leq X/Q$. Set $\tau=1-2\pi i\beta X$ and $\beta=\Theta-\frac{a}{q}$ such that $|\beta|\leq \frac{1}{qQ}$ with $Q=X(\log X)^{-A}$ for some real $A>0$. Furthermore, let
\begin{align*}
a_k=\frac{ak}{(q,k)}, \quad\text{and}\quad q_k=\frac{q}{(q,k)}.   
\end{align*}
Then
\begin{align*}
|\Phi_f(\rho e(\Theta))|
\ll\frac{1}{|\tau|q^2}{\zeta(2)(-1)^{\omega(q)}X\log\log X\prod_{p|q}p}.
\end{align*}
\end{lemma}

\begin{proof}
Utilizing the definition of~\eqref{loggeneratingfunction}, we express
\begin{align}\label{phifirst}
\Phi_f(\rho e(\Theta))
=&\Phi_f\left(\rho e\left(\beta+\frac{a}{q}\right)\right)
=\sum_{k=1}^{N}\frac{1}{k}\sum_{n}f(n)e\left(\frac{akn}{q}\right)\exp(-kn\tau/X)+O\left(\frac{X}{N}\right)\nonumber\\
=&\sum_{k=1}^{N}\frac{1}{k}\left(\sum_{\substack{\ell=1\\(\ell,q_k)=1}}^{q_k}e\left(\frac{a_k\ell}{q_k}\right)\sum_{n\equiv\ell(\bmod{q_k})}f(n)\exp(-kn\tau/X)+O(q^{\varepsilon}_k)\right)+O\left(\frac{X}{N}\right),
\end{align}
for real $\varepsilon>0$. Let $N=\sqrt{X}$ then employing Lemma~\ref{functionprogressionlemma}, and Abel summation formula the inner sum in~\eqref{phifirst} can be expressed as 

\begin{align}\label{siegelwalfiszapplication}
\sum_{n\equiv\ell(\bmod{q_k})}f(n)\exp(-kn\tau/X)
=&\frac{k\tau}{X}\int_{2}^{\infty}\sum_{\substack{n\leq u\\n\equiv\ell(\bmod{q_k})}}f(n)\exp(-ku\tau/X) du\nonumber\\
=&\frac{k\tau}{X}\int_{2}^{\infty}\left(\frac{uA_f(u)}{\varphi(q_k)}+u\mathcal{C}(q_k)+O\left(\frac{\sqrt{u}\|f\|}{\varphi(q_k)(\log u)^A}\right)\right)\exp(-ku\tau/X) du.
\end{align}

By partial summation formula the error term of~\eqref{siegelwalfiszapplication} is bounded by
\begin{align*}
\ll\sum_{k\leq \sqrt{X}}\frac{\varphi(q_k)}{k}(1+|\beta|X)\frac{\sqrt{X}\|f\|}{(\log X)^{A+3}}
\ll \frac{\sqrt{X}\|f\|}{(\log X)^2}.
\end{align*}

Let $h:\mathbb{N}\to\mathbb{R}$ be a continuously  differentiable function. Then for any constant $c$, recall the identity 
\begin{align}\label{exponentialidentity}
c\int_{2}^{\infty} h(u)\exp(-cu)du
=\int_{2}^{\infty}h'(u)\exp(-cu)du.
\end{align}

Substituting $h(u)=u\mathcal{C}(q_k)$ into~\eqref{siegelwalfiszapplication}, we get
\begin{align*}
\frac{k\tau}{X}\int_{2}^{\infty} u\mathcal{C}(q_k)\exp(-ku\tau/X)du
=&\mathcal{C}(q_k)\int_{2}^{\infty}\exp(-ku\tau/X)du
\ll \frac{X}{k\tau}.
\end{align*}


Now we compute the integral arising from the main term of~\eqref{siegelwalfiszapplication}.
\begin{align*}
\frac{k\tau}{X}\int_{2}^{\infty}\frac{uA_f(u)}{\varphi(q_k)}\exp(-k\tau u/X)du
=&\frac{k\tau}{X}\left(\int_{2}^{\frac{X}{k}A_f(X)}+\int_{\frac{X}{k}A_f(X)}^{\infty}\right)\frac{uA_f(u)}{\varphi(q_k)}\exp(-k\tau u/X)du\\
=&\mathcal{I}_1(X)+\mathcal{I}_2(X).
\end{align*}

We begin with the integral $\mathcal{I}_2(X)$ which is the integral parts with $u\geq\frac{X}{k}A_f(X)$. In this case we recall Merten's Theorem~\cite[Theorem 1.2]{harman} and condition~(C.1) to estimate $A_f(u)\ll \log\log u$. Suppose that
\begin{align*}
& t=\log\log u,\; s=-u\exp(-k\tau u/X),\\
& dt=\frac{du}{u\log u},\; \mathrm{and}\; ds=\frac{k\tau}{X}u\exp(-k\tau u/X)du.
\end{align*}

Applying integration by parts on $\mathcal{I}_2(X)$ we bound it by
\begin{align*}
\mathcal{I}_2(X)
\ll&\frac{k\tau}{X}\int_{\frac{X}{k}A_f(X)}^{\infty}\frac{1}{\varphi(q_k)}u\log\log u\exp(-k\tau u/X)du\\
\ll&\int_{\frac{X}{k}A_f(X)}^{\infty}\frac{1}{\varphi(q_k)}\frac{\exp(-k\tau u/X)}{\log u} du
\ll\frac{X}{k\tau\log X}\int_{A_f(X)}^{\infty}\frac{e^{-v}}{\varphi(q_k)}dv
\ll\frac{X e^{-A_f(X)}}{k\tau\varphi(q_k)\log X}.
\end{align*}

Now we compute the term
\begin{align*}
\mathcal{I}_1(X)
=\frac{k}{X}\int_{2}^{\frac{X}{k\tau}A_f(X)}\frac{uA_f(u)}{\varphi(q_k)}\exp(-k\tau u/X)du
\ll\frac{XA_f(X)}{k\tau\varphi(q_k)}\int_{2}^{\frac{X}{k\tau}A_f(X)}e^{-v}dv
\ll\frac{XA_f(X)}{k\tau\varphi(q_k)}.
\end{align*}

For $k\leq \sqrt{X}$ and $|\beta|\leq \frac{1}{qQ}$ we have $e^{-k/X}e(k\beta)=1+O(X^{-\frac{1}{4}})$. Thus by combining the estimates of $\mathcal{I}_1(X)$ and $\mathcal{I}_2(X)$ we get
\begin{align}\label{beforefinal}
\frac{k\tau}{X}\int_{2}^{\infty}\frac{uA_f(u)}{\varphi(q_k)}\exp(-ku\tau/X)du
\ll\frac{XA_f(X)}{k\tau\varphi(q_k)}.
\end{align}

We can readily verify that the Ramanujan sum yields the constant term
\begin{align}\label{meanvalueconstantterm}
\sum_{k=1}^{\infty}\frac{1}{\varphi(q_k)k^2}\sum_{\substack{\ell=1\\(\ell,q_k)=1}}^{q_k} e\left(\frac{a_k\ell}{q_k}\right)
=\frac{1}{q^2}{\zeta(2)(-1)^{\omega(q)}\prod_{p|q}p}.
\end{align}

Using the bound of~\eqref{beforefinal} in~\eqref{phifirst}, we have
\begin{align*}
|\Phi_f(\rho e(\Theta))|
=\left|\Phi_f\left(\rho e\left(\beta+\frac{a}{q}\right)\right)\right|
\ll&\frac{1}{|\tau|}\sum_{k=1}^{\sqrt{X}}\frac{1}{\varphi(q_k)k^2}\sum_{\substack{\ell=1\\(\ell,q_k)=1}}^{q_k} e\left(\frac{a_k\ell}{q_k}\right)XA_f(X)\\
\ll &\frac{1}{|\tau|q^2}{\zeta(2)(-1)^{\omega(q)}XA_f(X)\prod_{p|q}p}.
\end{align*}
Finally, by condition~(C.1) we write
\begin{align*}
A_f(N)
=\int_{2}^{N}\frac{P_f(u)}{u^2}du
\ll\int_{2}^{N}\frac{1}{u\log u}du
\ll \log\log N,
\end{align*}
concluding the lemma.
\end{proof}

\begin{remark}
Since our goal is solely to establish the mean value estimate in Lemma~\ref{meanvaluelemma}, the only necessary condition is that $f$ is ``well-distributed". The use of a weaker error term in~\eqref{bombierietal} will have no impact on the main result. 
\end{remark}

\section{Preliminary Reduction}\label{section3}

We adopt the method established in~\cite{montgomeryvaughan} to derive a sharp estimate for $S_f(N,\Theta)$ on the minor arcs. Despite Theorem~\ref{expsumwithadditivecoefficientforr} being a pivotal tool in addressing the partition problem, we consider a broader class of additive functions than those in Theorem~\ref{maintheorem}. While Theorem~\ref{expsumwithadditivecoefficientforr} is proven for strongly additive functions, our proof extends to a wider range of additive functions, as discussed later.
\footnotetext{Note that in \S\ref{section3}-\ref{section4} the implicit constants depends on $C$.}

\medskip

Let $f:\mathbb{N}\to\mathbb{C}$ be a complex valued strongly additive function satisfying conditions~\eqref{primecondition},~\eqref{meanvaluebound}~\eqref{secondmoment}. Then for $q\leq N$ and $(a,q)=1$, we first establish that
\begin{align}\label{expsumwithadditivecoefficientforrational}
S_f(N,a/q)
\ll\frac{N\log\log N}{\log N}+\frac{N\log\log N}{\varphi(q)^{\frac{1}{2}}}+(Nq)^{\frac{1}{2}}(\log(2N/q))^{\frac{3}{2}}\log\log N.      
\end{align}

We will then use~\eqref{expsumwithadditivecoefficientforrational} to prove the main theorem. However, before delving into the proof of this theorem, we establish some necessary prerequisites.

\subsection{Reduction into bilinear form}
Our goal is to decompose the exponential sum $S_f(N,a/q)$ into a bilinear form. Let $a\in\mathbb{Z}$ and $q\in\mathbb{N}$ such that $(a,q)=1$. Then by Cauchy's inequality and~\eqref{secondmoment}
\begin{align*}
\sum_{n\leq N}f(n)\log\left(\frac{N}{n}\right)e(an/q)
\ll \left(\sum_{n\leq N}\log^2\left(\frac{N}{n}\right)\right)^{\frac{1}{2}}\left(\sum_{n\leq N}|f(n)|^2\right)^{\frac{1}{2}}
\ll N\log\log N.
\end{align*}

Thus
\begin{align}\label{beforebilinear}
S_f(N,a/q)\log N
\ll \left|\sum_{n\leq N}f(n)(\log n)e(na/q)\right|+N\log\log N.
\end{align}

Using the identity $\log n=\sum_{d|n}\Lambda(d)$, where $\Lambda(n)$ denotes the Von-Mangoldt function on~\eqref{beforebilinear}, we deduce the bilinear form
\begin{align}\label{mainbilinearexpression}
\sum_{mn\leq N}f(mn)\Lambda(m)e(mna/q).
\end{align}

Note that the above sum vanishes unless $m=p^k$ for any prime $p$ and $k\geq 1$. Our aim is to replace $f(mn)$ with $f(m)+f(n)$ in the aforementioned expression. Thus, we write
\begin{align*}
R(N)
\coloneqq\sum_{\substack{mn\leq N}}|f(mn)-(f(m)+f(n))|\Lambda(m).
\end{align*}

We only need to consider the case when $(m,n)>1$ because $f(mn)=f(m)+f(n)$ whenever $(m,n)=1$. Therefore, we can infer that the contribution of the terms in $R(N)$ for $(m,n)>1$ is negligible. First we consider the sum
\begin{align*}
R_1(N)
\coloneqq\sum_{\substack{mn\leq N\\(m,n)>1}}|f(mn)|\Lambda(m)
=\sum_{\substack{p^kn\leq N\\k\geq 1\\ p|n}}|f(p^kn)|\log p.
\end{align*}

Recall that $f(p^k)=f(p)$ for all $k\geq 1$. Let $p^j\|p^kn$ with $j-k\geq 1$. Then $p^{j-k}\|n$, and by Cauchy's inequality and partial summation, we can express
\begin{align*}
R_1(N)
\ll&\sum_{p,j\geq 2}\sum_{\substack{n'\leq \frac{N}{p^j}\\p\nmid n'}}(|f(p^j)|+|f(n')|) (j-1)\log p\\
\ll& N\sum_{p,j\geq 2}\frac{|f(p^j)|j\log p}{p^j}+\sum_{p,j\geq 2}\sum_{\substack{n'\leq \frac{N}{p^j}\\p\nmid n'}}j{|f(n')|\log p}\\
\ll& N\left(\sum_{p,j\geq 2}\frac{j^2(\log p)^2}{p^{\frac{3j}{4}}}\right)^{\frac{1}{2}}\left(\sum_{n\geq 2}\frac{1}{n^{\frac{5}{4}}}\right)^{\frac{1}{2}}+N\log\log N\left(\sum_{p,j\geq 2}\frac{j^2(\log p)^2}{p^{\frac{3j}{4}}}\right)^{\frac{1}{2}}\left(\sum_{n\geq 2}\frac{1}{n^{\frac{5}{4}}}\right)^{\frac{1}{2}}\\
\ll& N\log\log N.
\end{align*}

With a similar argument we can estimate
\begin{align*}
R_2(N)
\coloneqq\sum_{\substack{mn\leq N\\(m,n)>1}}|f(m)+f(n)|\Lambda(m)
\ll& N\sum_{p,k\geq 1}\frac{(\log p)|f(p^k)|}{p^{j+k}}
+N\sum_{p,j\geq 1}\frac{(\log p)|f(p^j)|}{p^{j+k}}\\
&+\sum_{p,k,j\geq 1}\log p\sum_{n'\leq \frac{N}{p^{j+k}}}|f(n')|
\ll N\log\log N.
\end{align*}

Thus $R(N)\leq R_1(N)+R_2(N)\ll N\log\log N$. The terms of the initial sum~\eqref{mainbilinearexpression} with $m=p^k$ and $k\geq 2$ are negligible, following the same reasoning as before. Therefore, we are left with the terms
\begin{align}\label{bpndefinition}
\sum_{pn\leq N}(f(p)+f(n))e(pna/q)\log p
=&\sum_{pn\leq N}f(p)e(pna/q)\log p+\sum_{pn\leq N}f(n)e(pna/q)\log p\nonumber\\
=&\mathcal{B}_{p}+\mathcal{B}_{n}.
\end{align}

To obtain~\eqref{expsumwithadditivecoefficientforrational}, it is suffices to show that
\begin{align}\label{bilinearform}
\mathcal{B}_{p}+\mathcal{B}_{n}
\ll N\log\log N+\frac{N\log N\log\log N}{\varphi(q)^{\frac{1}{2}}}+(Nq)^{\frac{1}{2}}(\log(2N/q))^{\frac{3}{2}}\log N\log\log N.
\end{align}

To derive the above expression, we initially partition the sum over pairs $(p,n)$ into three components and evaluate the contribution of each part.

\subsection{Partition into rectangles}\label{rectanglepartition}
We divide the sum over pairs $(p,n)$ within rectangular regions $(P',P'']\times(N',N'']$ into three segments and determine the contribution of~\eqref{bpndefinition} in each segment. The subdivision of the rectangles has been adapted from~\cite[\S3]{montgomeryvaughan}, as we are dealing with a similar bilinear form over the region $\{(p,n):pn\leq N\}$. For clarity in the subsequent argument, we refer to~\cite[Figure 1]{lu}, which illustrates the following partition of regions.

\medskip

Define the rectangular segments $\mathcal{R}_{i}$ of the form
\begin{align*}
\mathcal{R}_{i}
=(0,2^{i}]\times\left(\frac{N}{2^{i+1}},\frac{N}{2^i}\right], \end{align*}

where $0\leq i\leq \log_2N=\frac{\log N}{\log 2}$. Then the remaining regions of $\left\{(p,n):pn\leq N\right\}$ are defined by
\begin{align*}
\mathcal{I}_{i}
=\left\{(x,y):xy\leq N,x>2^{i},\frac{N}{2^{i+1}}\leq y\leq\frac{N}{2^{i}}\right\}.
\end{align*}

Let
\begin{align*}
J_{i}
\coloneqq\min\{i+1,\lfloor{\log_2N\rfloor}-i+1,\lfloor{1/2\log_2(64N)\rfloor}\}.
\end{align*}

Now we place a series of rectangles $\mathcal{R}_{ijk}$ into the region $\mathcal{I}_{i}$ for $j=1,2,\ldots, J_{i}$ and $2^{j-1}<k\leq 2^{j}$, with $\mathcal{R}_{ijk}$ defined iteratively. Initially, we set
\begin{align*}
\mathcal{R}_{i12}
=\left(2^{i},\frac{2^{i+2}}{3}\right]\times\left(\frac{N}{2^{i+1}},\frac{3N}{2^{i+2}}\right]
\end{align*}

for each $0\leq i\leq \log_2N$. Notice that within each $\mathcal{I}_{i}$, we have two remaining regions after excluding the rectangle $\mathcal{R}_{i12}$. By repeating the same iterative process for $1\leq j\leq J_{i}$, we obtain a collection of $1+2\cdots+2^{j-1}$ rectangles, each taking the form

\begin{align*}
\mathcal{R}_{ijk}
=\left(\frac{2^{i+j}}{k},\frac{2^{i+j+1}}{2k-1}\right]\times\left(\frac{N(k-1)}{2^{i+j}},\frac{N(2k-1)}{2^{i+j+1}}\right].
\end{align*}

Thus we can see that each $\mathcal{R}_{ijk}$ is of the form $(P',P'']\times(N',N'']$ with
\begin{align*}
P''-P'\geq\frac{1}{4},\quad
N''-N'\geq\frac{1}{4},\quad
(P''-P')(N''-N')\gg \frac{2^{i+j}}{k(2k-1)}\cdot\frac{N}{2^{i+j+1}}\gg q.
\end{align*}

Then we partition the pair $(p,n)$ as
\begin{align*}
\{(p,n):pn\leq N\}
=\left(\bigcup_{i}\mathcal{R}_{i}\right)\bigcup\left(\bigcup_{i,j,k}\mathcal{R}_{ijk}\right)\bigcup\mathcal{E},
\end{align*}
where $\mathcal{E}$ does not lie in the rectangles $\mathcal{R}_{i}$ and $\mathcal{R}_{ijk}$. Write $\mathcal{R}_{i}
=\mathcal{P}_{i}\times\mathcal{N}_{i}$ with

\begin{align*}
&\mathcal{N}_i=\left(\frac{N}{2^{i+1}},\frac{N}{2^i}\right],\quad\text{and}\\
&\mathcal{H}_{i}
=\{(p,n)\in\mathcal{E}:n\in\mathcal{N}_{i}\}. 
\end{align*}

Then $\mathcal{E}$ is the union of $\mathcal{E}_{1}$, $\mathcal{E}_{2}$, and $\mathcal{E}_{3}$, which are the unions of those $\mathcal{H}_{i}$ with
\begin{align}\label{jcasedefinition}
J_{i}
=\begin{cases}
i+1 & \text{if $2^{2i}\leq N$, $2^{2i}\leq 16N/q$},\\
[\log_2N]-i+1 & \text{if $2^{2i}> N$, $2^{2i}>Nq/16$},\\
\frac{1}{2}\log_2(64N/q) & \text{if $16N/q<2^{2i}<Nq/16$}.
\end{cases}
\end{align}
respectively. We conclude this section with the following lemma, which provides an estimate for the contribution of the sum~\eqref{bpndefinition} when $(p,n)\in\mathcal{E}$.

\begin{lemma}\label{mathcalelemma}
Consider $(p,n)\in\mathcal{E}=\bigcup_{\ell=1,2,3}\mathcal{E}_{\ell}$ defined as above. Then the following upper bound holds
\begin{align*}
\mathcal{B}_{p,\mathcal{E}}+\mathcal{B}_{n,\mathcal{E}}
\ll& N\log\log N+(Nq)^{\frac{1}{2}}(\log\log (Nq))(\log 2N/q)^{\frac{1}{2}}\log N,
\end{align*}
where $\mathcal{B}_{p}$ and $\mathcal{B}_{n}$ are defined in~\eqref{bpndefinition}.
\end{lemma}

\begin{proof}
We proceed to evaluate the contributions of each sum $\mathcal{B}_{p}$ and $\mathcal{B}_{n}$ over the intervals $\mathcal{E}_{\ell}$ for $\ell=1,2,3$. According to~\eqref{jcasedefinition}, for $(p,n)\in\mathcal{E}_1$, we have $J_i=i+1$ and $2^{2i}\leq 16N/q$. By the definition of $\mathcal{E}_1$, for each prime $p$, the number of $n$ such that $(p,n)\in\mathcal{E}_1$ is bounded by $O(N/p^2)$.
Similarly, for each $n$, there are at most $O(1)$ primes $p$ for which $(p,n)\in\mathcal{E}_1$. Furthermore, applying the condition~\eqref{primecondition}, we observe that the second moment $\sum_{p\leq N}|f(p)|^2\log p$ over primes is bounded by $N$. 

\medskip

Thus, by employing Cauchy's inequality, partial summation, and condition~\eqref{secondmoment}, we obtain
\begin{align*}
\mathcal{B}_{p,\mathcal{E}_1}+\mathcal{B}_{n,\mathcal{E}_1}
\ll&\sum_{(p,n)\in\mathcal{E}_1}f(p)e(pna/q)\log p
+\sum_{(p,n)\in\mathcal{E}_1}f(n)e(pna/q)\log p\\
\ll&\left(\sum_{(p,n)\in\mathcal{E}_1}1\right)^{\frac{1}{2}} \left(\sum_{(p,n)\in\mathcal{E}_1}|f(p)\log p|^2\right)^{\frac{1}{2}}+\left(\sum_{(p,n)\in\mathcal{E}_1}|f(n)|^2\right)^{\frac{1}{2}}\left(\sum_{(p,n)\in\mathcal{E}_1}{\log^2p}\right)^{\frac{1}{2}}\\
\ll&\left(\sum_{n\leq N} 1\right)^{\frac{1}{2}}\left(N\sum_{p\leq N}\frac{|f(p)\log p|^2}{p^2}\right)^{\frac{1}{2}}+\left(\sum_{n\leq N} |f(n)|^2\right)^{\frac{1}{2}}\left(N\sum_{p\leq N}\frac{\log^2p}{p^2}\right)^{\frac{1}{2}}\\
\ll& N+N\log\log N
\ll N\log\log N.
\end{align*}

Note that the counting method applied to the pair $(p,n)\in\mathcal{E}_1$ was adopted from the proof of Lemma 2.2 in~\cite{lu}, wherein Jiang et al presented a detailed argument. Given the similarity in the counting techniques for these pairs in the intervals $\mathcal{E}_{\ell}$ for $\ell=1,2,3$, we employ their approach to show the counting details for the case where $(p,n)\in\mathcal{E}_2$ and $(p,n)\in\mathcal{E}_3$ will follow similarly.


\medskip

For each pair $(p,n)\in\mathcal{E}_2$, we have $n\leq\sqrt{2N}$ and $J_i=[\log_2N]-i+1$, if $2^{2i}>N$ and $2^{2i}>Nq/16$. Let
\begin{align*}
p\in(a,b]=\left(\frac{2^{i+J_i}}{k},\frac{2^{i+J_i+1}}{(2k-1)}\right]    
\end{align*}
generated after the $J_i$-th iteration and $2^{J_i}-1<k\leq 2^{J_i}$. Then we have two cases. When $``a"$ is generated by the last $J_i$-th iteration, i.e., $a=b=2^{i+J_i+1}/(2k-1)=4N/(2k-1)$. So for any fixed such $p$ it is evident that the number of $n$ for which  $(p,n)\in\mathcal{E}_2$ is at most $O(1)$. A similar argument follows when $a\ne b$.

\medskip

Similarly, for each pair $(p,n)\in\mathcal{E}_2$, we have $n\leq\sqrt{2N}$ and $n\in(N/2^{i+1},N/2^{i}]$. This indicates that for any fixed $n$, the number of primes $p$ such that $(p,n)\in\mathcal{E}_2$ lies within an interval of length $Nn^{-2}$. The intervals with $n\in(N/2^{i+1},N/2^{i}]$ can be partitioned into $2^{J_i}$ sub-intervals, which yields two cases, similar to before. In the first case
\begin{align*}
n\in\left(\frac{(k-1)N}{2^{i+J_i}},\frac{(2k-1)N}{2^{i+J_i+1}}\right].    
\end{align*}

Therefore, for any fixed $n$ satisfying the previous argument, there exists some $h$ for which, by the Brun-Titchmarsh theorem~\cite[Theorem 3.7]{halberstam}, the number of primes $p$ such that $(p,n)\in\mathcal{E}_2$ is at most"
\begin{align*}
\sum_{\substack{p\leq 4N/n^2\\p\equiv h(\bmod{q})}} 1
\ll\frac{N}{n^2(\log 4Nn^{-2})}.
\end{align*}

Thus, by applying Cauchy's inequality and partial summation, we estimate


\begin{align*}
\mathcal{B}_{p,\mathcal{E}_2}+\mathcal{B}_{n,\mathcal{E}_2}
\ll&\sum_{(p,n)\in\mathcal{E}_2}f(p)e(pna/q)\log p
+\sum_{(p,n)\in\mathcal{E}_2}f(n)e(pna/q)\log p\\
\ll&\left(\sum_{(p,n)\in\mathcal{E}_2}1\right)^{\frac{1}{2}}\left(\sum_{(p,n)\in\mathcal{E}_2}|f(p)\log p|^2\right)^{\frac{1}{2}}+\left(\sum_{(p,n)\in\mathcal{E}_2}|f(n)|^2\right)^{\frac{1}{2}}\left(\sum_{(p,n)\in\mathcal{E}_2}\log^2p\right)^{\frac{1}{2}}\\
\ll&\left(N\sum_{n\leq\sqrt{2N}}\frac{1}{n^2\log(4Nn^{-2})}\right)^{\frac{1}{2}}\left(\sum_{p\leq N}|f(p)\log p|^2\right)^{\frac{1}{2}}\\
&\quad\quad\quad\quad\quad\quad\quad\quad\quad\quad\quad\quad+\left(N\sum_{n\leq\sqrt{2N}}\frac{|f(n)|^2}{n^2\log(4Nn^{-2})}\right)^{\frac{1}{2}}\left(\sum_{p\leq N}\log^2p\right)^{\frac{1}{2}}\\
\ll& N+N\log\log N
\ll N\log\log N.
\end{align*}

With a similar reasoning as above for $(p,n)\in\mathcal{E}_3$, we have primes $p$ satisfying $\sqrt{N/q}\leq p\leq \sqrt{Nq}$. Furthermore, each prime $p$ for which $(p,n)\in\mathcal{E}_3$ lies in an interval of length $\sqrt{Nq}n^{-1}$. Applying  Brun-Titchmarsh theorem once more, we find that there are at most $O(\sqrt{Nq}n^{-1}(\log (2Nqn^{-2}))^{-1})$ such primes. Thus, we can conclude that

\begin{align*}
\mathcal{B}_{p,\mathcal{E}_3}+\mathcal{B}_{n,\mathcal{E}_3}
\ll&\sum_{(p,n)\in\mathcal{E}_3}f(p)e(pna/q)\log p
+\sum_{(p,n)\in\mathcal{E}_3}f(n)e(pna/q)\log p\\
\ll&\left(Nq\right)^{\frac{1}{2}}\left(\sum_{(N/q)^{\frac{1}{2}}\leq n\leq (Nq)^{\frac{1}{2}}}\frac{\log(2N/n)}{n\log(2Nqn^{-2})}\right)^{\frac{1}{2}}\left(\sum_{(N/q)^{\frac{1}{2}}\leq p\leq (Nq)^{\frac{1}{2}}}\frac{|f(p)|\log p}{p}\right)^{\frac{1}{2}}\\
&+\left(Nq\right)^{\frac{1}{2}}\left(\sum_{(N/q)^{\frac{1}{2}}\leq n\leq (Nq)^{\frac{1}{2}}}|f(n)|^2\frac{\log(2N/n)}{n\log(2Nqn^{-2})}\right)^{\frac{1}{2}}\left(\sum_{(N/q)^{\frac{1}{2}}\leq p\leq (Nq)^{\frac{1}{2}}}\frac{\log p}{p}\right)^{\frac{1}{2}}\\
\ll& (Nq)^{\frac{1}{2}}(\log 2N/q)^{\frac{1}{2}}\log q+(Nq)^{\frac{1}{2}}(\log 2N/q)^{\frac{1}{2}}(\log\log (Nq))\log q\\
\ll& (Nq)^{\frac{1}{2}}(\log 2N/q)^{\frac{1}{2}}(\log\log (Nq))\log q.
\end{align*}

Note that in the aforementioned expression, the ratio of logarithms $\log(2N/n)/\log(2Nqn^{-2})$ within the interval $[(N/q)^{\frac{1}{2}},(Nq)^{\frac{1}{2}}]$ is bounded by $\ll \log(4N/q)$. Additionally, we can simplify
\begin{align*}
&\log(\sqrt{Nq})(\log\log(\sqrt{Nq}))^2-\log(\sqrt{N/q})(\log\log(\sqrt{N/q}))^2\\
\ll&(\log(\sqrt{Nq})-\log(\sqrt{N/q}))(\log\log(\sqrt{Nq}))^2
\ll\log q(\log\log(Nq))^2.
\end{align*}

By combining the preceding estimates with $q\leq N$, we obtain
\begin{align*}
\mathcal{B}_{p,\mathcal{E}}+\mathcal{B}_{n,\mathcal{E}}
\ll& N\log\log N+(Nq)^{\frac{1}{2}}(\log\log (Nq))(\log 2N/q)^{\frac{1}{2}}\log N,
\end{align*}
thereby completing the lemma.
\end{proof}

\section{Proof of Theorem~\ref{expsumwithadditivecoefficientforr}}\label{section4}

In this section, we finalize the proof of Theorem~\ref{expsumwithadditivecoefficientforr}. It is worth noting that Jiang, L\"{u} and Wang~\cite{lu} have eased the requirement in condition\eqref{secondmoment} to
\begin{align}\label{sieveonshiftedprime}
\sum_{\substack{p\leq N\\p+h\in\mathbb{P}}}|f(p)f(p+h)|
\ll\frac{hN}{\varphi(h)(\log N)^2}
\end{align}
for any positive integer $h$~\cite[Theorem 3.11]{halberstam}. This adjustment is particularly significant in the study of $GL_m$ $L$-functions, especially in their application in the absence of progress towards the Ramanujan conjectures. However, relaxing condition~\eqref{secondmoment} will impact our results, as it allows for all strongly additive functions with normal orders, even those that cannot be determined using the Tur\'{a}n-Kubilius inequality.

Now, we establish a more general lemma, following the approach outlined in~\cite[\S4]{montgomeryvaughan}, which will enable us to achieve additional savings over the classical estimate of the exponential sum
\begin{align}\label{classicalexpsumestimate}
\sum_{n\leq N}e(\Theta n)\leq \min\left(N,\frac{1}{\|\Theta\|}\right).    
\end{align}

\begin{lemma}\label{weightlemma}
For each $k\leq K$, we define the rectangles $\mathcal{R}(k)=\mathcal{Q}(k)\times\mathcal{M}(k)$, where $\mathcal{Q}(k)$ and $\mathcal{M}(k)$ are disjoint intervals given by
\begin{align*}
\mathcal{Q}(k)=(Q'(k),Q''(k)],\quad
\mathcal{M}(k)=(M'(k),M''(k)],
\end{align*}
satisfying the conditions
\begin{align*}
&(Q'(k),Q''(k)]\subset (0,Q),\quad Q''(k)-Q'(k)\leq X,\\
&(M'(k),M''(k)]\subset (0,M),\quad M''(k)-M'(k)\leq Y,
\end{align*}
for some parameters $Q,M,X,Y$. Define
\begin{align*}
I(k)
\coloneqq \mathcal{B}_{p,\mathcal{R}(k)}+\mathcal{B}_{n,\mathcal{R}(k)}
=\sum_{k\leq K}\left(\sum_{(p,n)\in\mathcal{R}(k)}f(p)e(pna/q)\log p+\sum_{(p,n)\in\mathcal{R}(k)}f(n)e(pna/q)\log p\right).   
\end{align*}
Subsequently, under the conditions $(a,q)=1$ and $q\leq XY$, the following estimate holds
\begin{align*}
I(k)
\ll\log\log M\left(MQY\log2Q+MQXY/\varphi(q)+MQX+MQq\log(2XY/q)\right)^{\frac{1}{2}}.
\end{align*}
\end{lemma}

\begin{proof}
Let $\mathcal{M}(k)$ and $\mathcal{Q}(k)$ be the rectangles as defined above. Then by Cauchy's inequality we write
\begin{align}\label{idivision}
&\sum_{(p,n)\in\mathcal{R}(k)}f(p)e(pna/q)\log p+\sum_{(p,n)\in\mathcal{R}(k)}f(n)e(pna/q)\log p\nonumber\\
\ll&\left(\sum_{n\in\mathcal{M}(k)}1\right)^{\frac{1}{2}}\left(\sum_{n\in\mathcal{M}(k)}\left|\sum_{p\in\mathcal{Q}(k)}f(p)e(pna/q)\log p\right|^2\right)^{\frac{1}{2}}\nonumber\\
&\quad\quad\quad\quad+\left(\sum_{n\in\mathcal{M}(k)}|f(n)|^2\right)^{\frac{1}{2}}\left(\sum_{n\in\mathcal{M}(k)}\left|\sum_{p\in\mathcal{Q}(k)}e(pna/q)\log p\right|^2\right)^{\frac{1}{2}}\nonumber\\
\coloneqq&\mathcal{I}_{1,p}^{\frac{1}{2}}\mathcal{I}_{2,p}^{\frac{1}{2}}+\mathcal{I}_{1,n}^{\frac{1}{2}}\mathcal{I}_{2,n}^{\frac{1}{2}}.
\end{align}

Observe that the definition of $\mathcal{M}(k)$ and condition~\eqref{secondmoment} imply that the terms $\mathcal{I}_{1,p}$ and $\mathcal{I}_{1,n}$ can be estimated trivially. As argued in Montgomery-Vaughan~\cite[\S4]{montgomeryvaughan}, we introduce the following weight function to achieve a logarithmic saving compared to the classical estimate~\eqref{classicalexpsumestimate} for the remaining terms in~\eqref{idivision}. Let $w(n)$ be the weight function defined as
\begin{align}\label{weightfunctiondefinition}
w(n)=\max\left(0,2-\frac{|2n-2M'-Y|}{Y}\right).  
\end{align}

Note that $w(n)\geq 1$ for $n\in\mathcal{M}=(M',M'']$, which implies $M''-M'\leq Y$ by definition. Hence, the terms $\mathcal{I}_{2,p}$ and $\mathcal{I}_{2,n}$ can be estimated as

\begin{align}\label{i2p}
\mathcal{I}_{2,p}
\ll\sum_{n}w(n)\left|\sum_{p\in\mathcal{Q}}f(p)e(pna/q)\log p\right|^2
=\sum_{p,p'\in\mathcal{Q}} f(p)\Bar{f}(p')(\log p)(\log p')\sum_{n}w(n)e((p-p')na/q)
\end{align}

and
\begin{align}\label{i2n}
\mathcal{I}_{2,n}
\ll\sum_{n}w(n)\left|\sum_{p\in\mathcal{Q}}e(pna/q)\log p\right|^2
=\sum_{p,p'\in\mathcal{Q}}(\log p)(\log p')\sum_{n}w(n)e((p-p')na/q).
\end{align}

For the inner sum over $n$ in~\eqref{i2p} and~\eqref{i2n}, we employ the Euler-Maclaurin formula and partial summation (see~\cite[Lemma 2.3]{lu}) to obtain
\begin{align*}
\sum_{n}w(n)e((p-p')na/q)
=&\int_{M'-Y/2}^{M'+3Y/2} w(t)\left(e(t(p-p')a/q)+e(t(p-p')a/q)+t)\right)dt \\
&\quad\quad\quad\quad+O\left(\int_{M'-Y/2}^{M'+3Y/2}Y^{-1}(1+\|t\|)^{-1}dt\right)\\
\ll&\frac{1}{|(p-p')a/q|}\int_{M'-Y/2}^{M'+3Y/2} \frac{\partial}{\partial t}w(t) e(t(p-p')a/q)dt+O(1)\\
\ll&\frac{1}{Y|(p-p')a/q|}\int_{M'-Y/2}^{M'+3Y/2} d(e(t(p-p')a/q))+O(1)\\
\ll&\frac{1}{Y\|(p-p')a/q\|^2}+O(1).
\end{align*}

Combining the above estimate with~\eqref{classicalexpsumestimate}, we arrive at
\begin{align*}
\mathcal{I}_{2,p}
\ll \sum_{p,p'\in\mathcal{Q}}|f(p)\Bar{f}(p')|(\log p)(\log p')\min\left(Y,\frac{1}{Y\|(p-p')a/q\|^2}\right)
\end{align*}

and the bound for $\mathcal{I}_{2,n}$ follows similarly. Applying Cauchy's inequity yields
\begin{align*}
I(k)
=&\mathcal{B}_{p,\mathcal{R}(k)}+\mathcal{B}_{n,\mathcal{R}(k)}
\ll\left(\sum_{k\leq K}\sum_{n\in\mathcal{M}(k)}1\right)^{\frac{1}{2}}\left(\sum_{k\leq K}\sum_{n\in\mathcal{M}(k)}\left|\sum_{p\in\mathcal{Q}(k)}f(p)e(pna/q)\log p\right|^2\right)^{\frac{1}{2}}\nonumber\\
&\quad\quad\quad\quad\quad\quad\quad+\left(\sum_{k\leq K}\sum_{n\in\mathcal{M}(k)}|f(n)|^2\right)^{\frac{1}{2}}\left(\sum_{k\leq K}\sum_{n\in\mathcal{M}(k)}\left|\sum_{p\in\mathcal{Q}(k)}e(pna/q)\log p\right|^2\right)^{\frac{1}{2}}\\
\ll& M^{\frac{1}{2}}\left(\sum_{k\leq K}\sum_{p,p'\in\mathcal{Q}(k)}f(p)\Bar{f}(p')(\log p)(\log p')\min\left(Y,\frac{1}{Y\|(p-p')a/q\|^2}\right)\right)\\
&+ M^{\frac{1}{2}}\log\log M\left(\sum_{k\leq K}\sum_{p,p'\in\mathcal{Q}(k)}(\log p)(\log p')\min\left(Y,\frac{1}{Y\|(p-p')a/q\|^2}\right)\right)\\
\ll& M^{\frac{1}{2}}\left(QY\log Q+\log^2Q\sum_{0<h\leq X}\sum_{\substack{p\leq Q\\p+h=p'}}|f(p)\Bar{f}(p')|\min\left(Y,\frac{1}{Y\|ha/q\|^2}\right)\right)^{\frac{1}{2}}\\
&+M^{\frac{1}{2}}\log\log M\left(QY\log Q+\log^2Q\sum_{0<h\leq X}\sum_{\substack{p\leq Q\\p+h=p'}}\min\left(Y,\frac{1}{Y\|ha/q\|^2}\right)\right)^{\frac{1}{2}}.
\end{align*}

Considering condition~\eqref{primecondition} for all primes $p$ and the sieve estimate~\eqref{sieveonshiftedprime}, the innermost sum is bounded by $\ll hQ(\log 2Q)^{-2}\varphi(h)^{-1}$. Thus, using the identity $h\varphi(h)^{-1}\ll \sum_{m|h}\frac{1}{m}$, we get

\begin{align}\label{iklast}
I(k)
\ll&\left(MQY\log Q+MQ\sum_{0<h\leq X}\frac{h}{\varphi(h)}\min\left(Y,\frac{1}{Y\|ha/q\|^2}\right)\right)^{\frac{1}{2}}\nonumber\\
&+\left(MQY(\log\log M)^2\log Q+MQ(\log\log M)^2\sum_{0<h\leq X}\frac{h}{\varphi(h)}\min\left(Y,\frac{1}{Y\|ha/q\|^2}\right)\right)^{\frac{1}{2}}\nonumber\\
\ll&\left(MQY\log Q+MQ\sum_{m\leq X}\frac{1}{m}\sum_{\ell\leq X/m}\min\left(Y,\frac{1}{Y\|m\ell a/q\|^2}\right)\right)^{\frac{1}{2}}\nonumber\\
&+\left(MQY(\log\log M)^2\log Q+MQ(\log\log M)^2\sum_{m\leq X}\frac{1}{m}\sum_{\ell\leq X/m}\min\left(Y,\frac{1}{Y\|m\ell a/q\|^2}\right)\right)^{\frac{1}{2}}.
\end{align}

Let
\begin{align}\label{vdefinition}
V=\sum_{m\leq X}\frac{1}{m}\sum_{\ell\leq X/m}\min\left(Y,\frac{1}{Y\|m\ell a/q\|^2}\right).    
\end{align}

Take $a_m=\frac{am}{(q,m)}$ and $q_m=\frac{q}{(q,m)}$, ensuring $(a_m,q_m)=1$. Hence, the inner sum of~\eqref{vdefinition} is bounded by
\begin{align*}
\sum_{\ell\leq X/m}\min\left(Y,\frac{1}{Y\|\ell a_m/q_m\|^2}\right)
\ll \min\left(\frac{XY}{m},\left(\frac{X}{mq_m}+1\right)\left(Y+q_m\right)\right).
\end{align*}

Thus we arrive at
\begin{align*}
V=\sum_{m\leq X}\frac{1}{m}\sum_{\ell\leq X/m}\min\left(Y,\frac{1}{Y\|m\ell a/q\|^2}\right)
\ll&\sum_{\substack{m\leq X\\XY/m\leq q_m}}\frac{XY}{m^2}+\sum_{\substack{m\leq X\\XY/m> q_m}}\frac{1}{m}\left(\frac{XY}{mq_m}+\frac{X}{m}+Y+q_m\right)\\
\ll&\sum_{r|q}\sum_{s>XY/q}\frac{XY}{s^2r^2}+\sum_{r|q}\sum_{s}\frac{XY}{s^2rq}+X+Y\log X\\
&\quad\quad\quad\quad\quad\quad\quad\quad\quad\quad\quad\quad+\sum_{r|q}\sum_{s<XY/q}\frac{q}{sr^2}\\
\ll&\frac{XY}{\varphi(q)}+X+Y\log X+q\log\left(\frac{2XY}{q}\right). 
\end{align*}

Implementing the upper bound of $V$ in~\eqref{iklast} yields
\begin{align*}
I(k)
\ll&\left(MQY\log2Q+MQXY/\varphi(q)+MQX+MQq\log(2XY/q)\right)^{\frac{1}{2}}\\ 
&+\log\log M\left(MQY\log2Q+MQXY/\varphi(q)+MQX+MQq\log(2XY/q)\right)^{\frac{1}{2}}
\end{align*}
completing the lemma.
\end{proof}

We refer to Remark 2.1 in~\cite{lu}, which elucidates the reduction of the ``log" factor in the upper bound of $V$ arising due to the weight function $w$.

\subsection{Completion of the proof}
First we apply Lemma~\ref{weightlemma} when $(p,n)\in\mathcal{R}_{i}$. We choose $K=1$, $X=Q=2^{i}$, $Y=M=\frac{N}{2^{i}}$. Thus
\begin{align}\label{proofstep1}
\mathcal{B}_{p,\mathcal{R}_i}+\mathcal{B}_{n,\mathcal{R}_i}
=&\sum_{0\leq i\leq\frac{\log N}{\log 2}}\left(\sum_{(p,n)\in\mathcal{R}_i}f(p)e(pna/q)\log p+\sum_{(p,n)\in\mathcal{R}_i}f(n)e(pna/q)\log p\right)\nonumber\\
\ll&\sum_{0\leq i\leq\frac{\log N}{\log 2}}\log\log\left(\frac{N}{2^i}\right)\left(N\left(\frac{(i+1)}{2^{i}}\right)^{\frac{1}{2}}+\frac{N}{\varphi(q)^{\frac{1}{2}}}+(N2^{i})^{\frac{1}{2}}+(Nq\log(2N/q))^{\frac{1}{2}}\right)\nonumber\\
\ll& N\log\log N+\frac{N\log N\log\log N}{\varphi(q)^{\frac{1}{2}}}+\left(Nq\log(2N/q)\right)^{\frac{1}{2}}\log N\log\log N.
\end{align}

Next, we compute the contribution of $(p,n)\in\mathcal{R}_{ijk}$ as defined in Section~\ref{rectanglepartition}. In each $\mathcal{R}_{ijk}$, we have $2^{j-1} < k \leq 2^{j}$. We take
\begin{align*}
&K=2^{j-1},\;
Q=2^{i+1},\;
M=\frac{N}{2^{i}},\\
&X=2^{i-j+1},\;
Y=\frac{32N}{2^{i+j}}\quad\text{and}\quad
1\leq j\leq J_{i}.
\end{align*}

Hence, according to~\eqref{jcasedefinition}, we have $XY = \frac{64N}{2^{2j}} \geq \frac{64N}{2^{2J_i}} \geq q$. Using the chosen parameters in Lemma~\ref{weightlemma}, we obtain
\begin{align*}
&\sum_{2^{j-1}<k\leq 2^{j}}\left(\sum_{(p,n)\in\mathcal{R}_{ijk}}f(p)e(pna/q)\log p+\sum_{(p,n)\in\mathcal{R}_{ijk}}f(p)e(pna/q)\log p\right)\\ \ll&\log\log(N/2^i)\left(N\left(\frac{(i+1)}{2^{i+j}}\right)^{\frac{1}{2}}+\frac{N}{2^{j}\varphi(q)^{\frac{1}{2}}}+(N2^{i-j})^{\frac{1}{2}}+(Nq\log(2N/q))^{\frac{1}{2}}\right).
\end{align*}

Note that $(p,n)\in\mathcal{R}_{ijk}$ can be written as $(p,n)\in\bigcup_{i,j,k}\mathcal{R}_{i}$ and by~\eqref{jcasedefinition} $J_i\ll\log(2N/q)$. Summing over $1\leq j\leq J_i$ the above expression becomes
\begin{align*}
&\sum_{1\leq j\leq J_i}\log\log\left(\frac{N}{2^i}\right)\left(N\left(\frac{(i+1)}{2^{i+j}}\right)^{\frac{1}{2}}+\frac{N}{2^{j}\varphi(q)^{\frac{1}{2}}}+(N2^{i-j})^{\frac{1}{2}}+(Nq\log(2N/q))^{\frac{1}{2}}\right)\\
\ll& N\left(\frac{(i+1)(\log\log(N/2^i))^2}{2^{i}}\right)^{\frac{1}{2}}+\frac{N}{\varphi(q)^{\frac{1}{2}}}+(N2^{i})^{\frac{1}{2}}+(Nq)^{\frac{1}{2}}(\log(2N/q))^{\frac{3}{2}}.
\end{align*}

Now summing over $0\leq i\leq \frac{\log N}{\log 2}$ gives
\begin{align}\label{proofstep2}
&\sum_{0\leq i\leq\frac{\log N}{\log 2}}\log\log\left(\frac{N}{2^i}\right)\left(N\left(\frac{(i+1)}{2^{i}}\right)^{\frac{1}{2}}+\frac{N}{\varphi(q)^{\frac{1}{2}}}+(N2^{i})^{\frac{1}{2}}+(Nq\log(2N/q))^{\frac{1}{2}}\right)\nonumber\\
\ll&N\log\log N+\frac{N\log N\log\log N}{\varphi(q)^{\frac{1}{2}}}+(Nq)^{\frac{1}{2}}(\log(2N/q))^{\frac{3}{2}}\log N\log\log N.
\end{align}

By combining~\eqref{proofstep1} and~\eqref{proofstep2} with Lemma~\ref{mathcalelemma} we obtain~\eqref{bilinearform}, and therefore~\eqref{expsumwithadditivecoefficientforrational}. Building on the foundation of \eqref{expsumwithadditivecoefficientforrational}, we now complete the proof of Theorem~\ref{expsumwithadditivecoefficientforr} by employing the method outlined in~\cite[\S6]{montgomeryvaughan}.

\medskip

{\it Proof of Theorem~\ref{expsumwithadditivecoefficientforr}.} Let $\Theta$ be any real number satisfying~\eqref{dirichletconditionforalltheta}. Recalling the definition~\eqref{expsumadditivedefinition} we write
\begin{align*}
S_f(N,\Theta)
=e((\Theta-\beta)N)S_f(N,\beta)-2\pi i(\Theta-\beta)\int_{1}^{N}S_f(u,\beta)e((\Theta-\beta)u)du.
\end{align*}

Let $\beta=b/r$ where $(b,r)=1$ and $r\leq N$. Then we apply condition~\eqref{meanvaluebound} when $u\leq r$, and~\eqref{expsumwithadditivecoefficientforrational} when $u>r$, yielding
\begin{align}\label{withthetabr}
S_f(N,\Theta)
\ll\frac{N\log\log N}{\log N}+\frac{N\log\log N}{\varphi(q)^{\frac{1}{2}}}+(Nr)^{\frac{1}{2}}(\log(2N/r))^{\frac{3}{2}}\log\log N\left(1+N\left|\Theta-\frac{b}{r}\right|\right).
\end{align}

Now, we consider two cases. If $q>N^{\frac{1}{2}}$, we set $b=a$ and $r=q$, obtaining
\begin{align*}
S_f(N,\Theta)
\ll\frac{N\log\log N}{\log N}+\frac{N\log\log N(\log R)^{\frac{3}{2}}}{R^{\frac{1}{2}}}
\end{align*}
for $2\leq R\leq q\leq N/R$ and $|\Theta-a/q|\leq q^{-2}$. In the second case, when $q\leq N^{\frac{1}{2}}$, Dirichlet's theorem ensures the existence of $b,r$ such that $(b,r)=1$ and 
\begin{align*}
\left|\Theta-\frac{a}{q}\right|\leq \frac{1}{2Nr/q}, \quad r\leq\frac{2N}{q}.
\end{align*}

Consequently, either $r=q$ or,
\begin{align*}
\left|a-\frac{bq}{r}\right|
=q\left|\left(\Theta-\frac{b}{r}\right)-\left(\Theta-\frac{a}{q}\right)\right|
\leq \frac{q^2}{2Nr}+\frac{1}{q}
\leq \frac{1}{2}+\frac{1}{q}.
\end{align*}

Since $1\leq \left|a-\frac{bq}{r}\right|$, in either case it follows that $r\geq\frac{1}{2}$. Hence, $\left|\Theta-\frac{b}{r}\right|\leq\frac{1}{N}$ and by invoking~\eqref{withthetabr}, we complete the proof.
\qed

\medskip

Corollary~\ref{expsumwithadditivecoefficient} follows from Theorem~\ref{expsumwithadditivecoefficientforr} and the argument presented in the proof of Corollary 2 in~\cite[\S6]{montgomeryvaughan}.

\begin{remark}
Let $f:\mathbb{N}\to\mathbb{C}$ be any complex valued additive functions. For $N\geq 2$, the mean $\mu_f$ and variance $\nu_f$ of $f$ is defined as follows
\begin{align}\label{meanvarianceoff}
\mu_f(N)\coloneqq\sum_{p^k\leq N} \frac{f(p^k)}{p^k}\left(1-\frac{1}{p}\right),\quad\text{and}\quad
\nu_f(N)\coloneqq\sum_{p^k\leq N} \left(\frac{|f(p^k)|^2}{p^k}\right)^{\frac{1}{2}}.
\end{align}

It is a well established fact, showed in~\cite[Lemma 3.1]{sacha}, that as $N\to\infty$, $\mu_f(N)$ represents the asymptotic mean value for $\{f(n)\}_{n\leq N}$. We consider the class $\mathcal{S}$ of additive functions satisfying the following conditions.
\begin{enumerate}
\item $\nu_f(N)\to\infty$ as $N\to\infty$.
\item $\nu_f(N)$ is dominated by its prime factor in the sense that
\begin{align*}
\mathop{\lim\sup}_{N\to\infty} \frac{1}{\nu_f(N)^2}\sum_{\substack{p^k\leq N\\k\geq 2}}\frac{|f(p^k)|^2}{p^k}=0.   
\end{align*}
\end{enumerate}
Note that the class $\mathcal{S}$ includes almost all completely and strongly additive functions (cf. Lemma 3.6a of~\cite{sacha}). Extending the proof of Theorem~\ref{expsumwithadditivecoefficientforrational} for the class $\mathcal{S}$ involves deriving analogous conditions to~\eqref{meanvaluebound} and~\eqref{secondmoment} using using~\eqref{meanvarianceoff}. For an in-depth analysis of the behavior of $\mu_f(N)$, we refer to Section 8.1 of~\cite{sacha} and~\cite[\S8]{elliottpart2}.
\end{remark}

\section{Minor arcs Analysis}\label{minorarcs}

\begin{lemma}\label{minorarclemma}
Suppose that $A>A_0$ is a positive real number such that $X>X_0(A)$. Set $Q = X(\log X)^{-A}$. Consider a real number $\Theta$ such that for all $a \in \mathbb{Z}$ and $q \in \mathbb{N}$ with $(a,q)=1$, and for $\left|\Theta-\frac{a}{q}\right|\leq \frac{1}{qQ}$ we have $q > X/Q$. Then for $\varepsilon>0$,
\begin{align*}
\Phi_f(\rho e(\Theta))\ll {X(\log X)^{-1+\varepsilon}},
\end{align*}
where the implicit constant depends on $A$.
\end{lemma}

\begin{proof}
In accordance with the expression~\eqref{loggeneratingfunction}, taking $\rho=e^{-1/X}$ we have
\begin{align*}
\Phi_f(\rho e(\Theta))
=\sum_{k=1}^{\infty}\sum_{n=1}^{\infty}\frac{f(n)}{k}e^{-nk/X}\exp(kn\Theta).
\end{align*}

Recall the identity
\begin{align*}
e^{-nk/X}
=\int_{n}^{\infty}kX^{-1}e^{-ku/X}du.
\end{align*}

Emphasizing this to $\Phi_f(\rho e(\Theta))$ for $\Theta\in\mathfrak{m}$ we obtain
\begin{align*}
\Phi_f(\rho e(\Theta))
=\sum_{k=1}^{\infty}\frac{1}{k}\int_{2}^{\infty}kX^{-1}e^{-ku/X}\sum_{n\leq u}f(n)e(kn\Theta)du.
\end{align*}

The above integrand can be crudely estimated by
\begin{align*}
\int_{2}^{\infty}kX^{-1}e^{-ku/X}\sum_{n\leq u}f(n)e(kn\Theta)du
\ll \int_{0}^{\infty}ukX^{-1}e^{-ku/X}du.
\end{align*}

Applying integration by parts for any $\delta>0$, we arrive at
\begin{align}\label{exponentialidentity2}
\int_{2}^{\infty} u^{\delta}kX^{-1}e^{-ku/X}du
\ll \left(\frac{X}{k}\right)^{\delta}.
\end{align}

Let $K=(\log X)^{A/2}$. Therefore, 
\begin{align*}
\sum_{k=K+1}^{\infty}\frac{1}{k}\int_{2}^{\infty}kX^{-1}e^{-ku/X}\sum_{n\leq u}f(n)e(kn\Theta)du
\ll X\sum_{k=K+1}^{\infty}\frac{1}{k^2}\ll\left(\frac{X}{K}\right).
\end{align*}

Now we focus on the terms for $k\leq K$. For any given $k$, we choose $a_k$ and $q_k$ such that $(a_k,q_k)=1$, $q_k\leq Q$, and $|\Theta k-a_k/q_k|\leq 1/qQ$. For $q_k\geq (\log X)^{A/2}$, by Corollary~\ref{expsumwithadditivecoefficient} and the identity~\eqref{exponentialidentity2}, we crudely estimates the minor arcs by

\begin{align*}
\Phi_f(\rho e(\Theta))
=&\sum_{k=1}^{K}\frac{1}{k}\int_{2}^{\infty}kuX^{-1}e^{-uk/X}\sum_{n\leq u}f(n)\exp(\Theta kn)du+O\left({X}{(\log X)^{-A/2}}\right)\\
\ll&\frac{X\log\log X}{\log X}\sum_{k\leq K}\frac{1}{k^2}+O\left({X}{(\log X)^{-A/2}}\right)\\
\ll&\frac{X\log\log X}{\log X}\left(\zeta(2)+O\left(\frac{1}{(\log X)^{A/2}}\right)\right)+O\left({X}{(\log X)^{-A/2}}\right)
\ll \frac{X\log\log X}{\log X}.
\end{align*}

Lastly, we examine for $k\leq K$ such that $q_k\leq (\log X)^{A/2}$. If such a $k$ exists, then $|\Theta-\frac{a_k}{kq_k}|\leq 1/kqQ$, $kq_k\leq X/Q$, and $a_k/kq_k=a/q$ for some $(a,q)=1$ with $q\leq kq_k$. However, this contradicts the hypothesis. Therefore, no such $k$ exists, concluding the lemma.
\end{proof}

\section{Proof of Theorem~\ref{maintheorem}}\label{maintheoremproof}

In this section, we establish the main result. As outlined in the introduction, the key element of our theorem is the saddle-point method, which we analyze first to determine the contribution arising from it, and then we derive our main result.

\subsection{The saddle-point solution} As described in~\cite[\S1]{vaughan} and~\cite[\S2]{debten}, the principle of the saddle-point method involves choosing $\rho=\rho(n)$ such that for every $n\geq 0$, the equation
\begin{align}\label{uniquesoltuion}
\rho\Phi_f'(\rho)=n
\end{align}
has a unique solution. Now observe that for the choice of our $\rho$ the function $-\Phi_f'(\rho)$ (as seen in~\eqref{loggeneratingfunction}) strictly decreases for $\rho(n)\to 1^{-}$ as $n\to\infty$. Therefore, by the mean value theorem, $-\Phi_f'(\rho)=n$ does have a unique solution.

\begin{lemma}\label{asymptoticestimate}
For $\rho=\rho(x)$, as $x\to\infty$ one has that
\begin{align*}
x\log\frac{1}{\rho(x)}
=\left({x}{\zeta(2)c_f(\log\log x+\psi_f)}\right)^{\frac{1}{2}}\left(1+O\left(\frac{1}{\log x}\right)\right).
\end{align*}
Furthermore,
\begin{align}\label{asymptoticequation}
\Phi_{f,(m)}(\rho(x))
=\Gamma(m+1)x^{\frac{m+1}{2}}\left(\frac{1}{{\zeta(2)c_f(\log\log x+\psi_f)}}\right)^{\frac{m-1}{2}}\left(1+O\left(\frac{1}{\log x}\right)\right),
\end{align}
where the constants $c_f$ and $\psi_f$ are given in~\eqref{constants}.
\end{lemma}

\begin{proof}
Assume $x$ is sufficiently large, and $\rho$ is determined by~\eqref{uniquesoltuion}. Suppose that $\rho=\rho(x)$ is very close to 1, and $X(x)$ is defined as $\rho(x)=e^{-{1}/{X(x)}}$, then $X(x)=\frac{1}{\log\frac{1}{\rho(x)}}$ will be large. Then by Lemma~\ref{majorarcsestimate}, we have
\begin{align*}
x=\rho\frac{d}{d\rho}\Phi_f(\rho)
=X(x)^2\zeta(2)\Gamma(2)c_f(\log\log X(x)+\psi_f)\left(1+O\left(\frac{1}{\log X(x)}\right)\right).
\end{align*}

Taking logarithm on both side we see that
\begin{align*}
\log X(x)
=\frac{1}{2}(\log x-\log\log\log x-\log(\zeta(2)\Gamma(2)c_f))\left(1+O\left(\frac{1}{\log x}\right)\right),
\end{align*}
where we have utilized the Taylor expansion for the natural logarithm and
\begin{align*}
\log\log\log x
=\log\log\log X(x)+O\left(\frac{1}{\log\log X(x)}\right).
\end{align*}

Solving for $X(x)$, we obtain
\begin{align}\label{bigxvalue}
X(x)
=\left(\frac{x}{\zeta(2)\Gamma(2)c_f(\log\log x+\psi_f)}\right)^{\frac{1}{2}}\left(1+O\left(\frac{1}{\log x}\right)\right).
\end{align}

Therefore, as argued in~\cite[\S3]{vaughan}, from~\eqref{bigxvalue} we deduce that
\begin{align*}
x\log\frac{1}{\rho(x)}
=\frac{x}{X(x)}
=&\frac{x}{\left(\frac{x}{\zeta(2)\Gamma(2)c_f(\log\log x+\psi_f)}\right)^{\frac{1}{2}}}\left(1+O\left(\frac{1}{\log x}\right)\right)\\
=&\left({x}{\zeta(2)\Gamma(2)c_f(\log\log x+\psi_f)}\right)^{\frac{1}{2}}\left(1+O\left(\frac{1}{\log x}\right)\right).
\end{align*}

From Lemma~\ref{majorarcsestimate}, we write
\begin{align*}
\Phi_{f,(m)}(\rho(x))
=\zeta(2)\Gamma(m+1)X(x)^{m+1}c_f(\log\log X(x)+\psi_f)\left(1+O\left(\frac{1}{\log X(x)}\right)\right)
\end{align*}

Substituting the value of $X(x)$ from~\eqref{bigxvalue} we arrive at
\begin{align*}
\Phi_{f,(m)}(\rho(x))
=&\Gamma(m+1)\zeta(2)c_f\left(\frac{x}{\zeta(2)c_f(\log\log x+\psi_f)}\right)^{\frac{m+1}{2}}(\log\log x+\psi_f)\left(1+O\left(\frac{1}{\log x}\right)\right)\\
=&\Gamma(m+1)x^{\frac{m+1}{2}}\left(\frac{1}{{\zeta(2)c_f(\log\log x+\psi_f)}}\right)^{\frac{m-1}{2}}\left(1+O\left(\frac{1}{\log x}\right)\right).
\end{align*}
\end{proof}

\begin{theorem}\label{secondarytheorem}
Let $\rho=\rho(n)$, then for $\varepsilon>0$
\begin{align*}
p_f(n)=\frac{\rho^{-n}\Psi_f(\rho)}{\sqrt{2\pi\Phi_{f,(2)}(\rho)}}(1+O(n^{-1+\varepsilon})).   
\end{align*}
\end{theorem}

\begin{proof}
Let $A > A_0$ for some positive real number $A_0$. Referring to the setup of arcs in Section~\ref{arcssetup}, along with Lemma~\ref{meanvaluelemma} and Lemma~\ref{minorarclemma} for $q > 1$, and Lemma~\ref{majorarcsestimate} for $q = 1$ and $a = 0$, it follows that if $|\Theta| \geq \tau$, where $\tau = (\log X)^{-1/4}X^{-1}$, then\footnotetext{Observe that by the choice of $|\Theta|\geq \tau$, $\mathrm{Re}((1+4\pi^2\Theta^2X^2)^{-1/2})$ is bounded above by $1-2\pi^2(\log X)^{-1/2}$.}
\begin{align*}
\mathrm{Re}\left(\Phi_f(\rho e(\Theta))\right)
\leq (1 -(\log X)^{-1})\Phi_f(\rho).
\end{align*}

Thus by Lemma~\ref{majorarcsestimate}, we write
\begin{align*}
\Psi_f(\rho e(\Theta))
\ll \Psi_f(\rho)n^{-C}
\end{align*}
for an arbitrarily large constant $C$. Furthermore, taking $m=0$ and $x=n$ in~\eqref{asymptoticequation} gives us
\begin{align*}
\Phi_f(\rho)
\sim n^{\frac{1}{2}}(\zeta(2)c_f(\log\log n+\psi_f))^{-\frac{1}{2}}, 
\end{align*}

where the constants $c_f$ and $\psi_f$ are defined in~\eqref{constants}. Then by triangle inequality, we write
\begin{align*}
\left|\int_{[-1/2,1/2]\backslash\mathfrak{M}(1,0)}\exp(\Phi_f(\rho e(\Theta)))e(-n\Theta)d\Theta\right| 
\ll (\Phi_f(\rho))^{-B_1}\exp(\Phi_f(\rho ))
\ll n^{-B_2}\exp(\Phi_f(\rho )),
\end{align*}
where $B_1,B_2>0$ are arbitrary constants. Thus by~\eqref{cauchyformula}
\begin{align}\label{beforetaylortheorem}
p_f(n)
=\rho^{-n}\int_{-\tau}^{\tau}\exp(\Phi_f(\rho e(\Theta)))e(-n\Theta)d\Theta+O\left(\frac{\rho^{-n}\Psi_f(\rho)}{n^{B_2}}\right).
\end{align}

Now it remains to study the integral in~\eqref{beforetaylortheorem}. Let $\beta$ be any real number, with $\mathrm{Re}(\beta)$ and $\mathrm{Im}(\beta)$ denoting the real and imaginary parts of $\Phi_f(\rho e(\beta))$, respectively. Applying Taylor's theorem with a remainder term, we have
\begin{align*}
\mathrm{Re}(\beta)
= \mathrm{Re}(0) + \beta\mathrm{Re}'(0) + \frac{1}{2!}\beta^2\mathrm{Re}''(\beta) + \frac{1}{3!}\beta^3\mathrm{Re}'''(\theta_\mathrm{Re}\beta),
\end{align*}
and
\begin{align*}
\mathrm{Im}(\beta) 
= \mathrm{Im}(0) + \beta\mathrm{Im}'(0) + \frac{1}{2!}\beta^2\mathrm{Im}''(\beta) + \frac{1}{3!}\beta^3\mathrm{Im}'''(\theta_\mathrm{Im}\beta),
\end{align*}
where $0<\theta_\mathrm{Re},\theta_\mathrm{Im}<1$. Now

\begin{align*}
&\mathrm{Re}'(\beta)+i\mathrm{Im}'(\beta)
=2\pi ie(\beta)\rho\Phi_f'(\rho e(\beta)),\\
&\mathrm{Re}''(\beta)+i\mathrm{Im}''(\beta)
=-4\pi^2e(\beta)\rho\Phi_f'(\rho e(\beta))-4\pi^2e(2\beta)\rho^2\Phi_f''(\rho e(\beta)),\\
&\mathrm{Re}'''(\beta)+i\mathrm{Im}'''(\beta)
=-8\pi^3ie(\beta)\rho\Phi_f'(\rho e(\beta))-24\pi^3ie(2\beta)\rho^3\Phi_f''(\rho e(\beta))-8\pi^3ie(3\beta)\rho^3\Phi_f'''(\rho e(\beta)).
\end{align*}

Hence, from the above expressions, we derive the following inequality for any real $\beta$.

\begin{align*}
\sup\left(|\mathrm{Re}'''(\beta)|,|\mathrm{Im}'''(\beta)|\right)
\leq 8\pi^3\rho\Phi_f'(\rho)+24\pi^3\rho^3\Phi_f''(\rho)+8\pi^3\rho^3\Phi_f'''(\rho).
\end{align*}

Thus
\begin{align*}
\Phi_f(\rho e(\beta))
=&\Phi_f(\rho)+\beta2\pi i\rho\Phi'_f(\rho)-\frac{1}{2}\beta^24\pi^2\left(\rho\Phi'_f(\rho)+\rho^2\Phi''_f(\rho)\right)\\
&+\frac{1}{3}\mathfrak{w}|\beta|^3\left(8\pi^3\rho\Phi'_f(\rho)+24\pi^3\rho^2\Phi''_f(\rho)+8\pi^3\rho^3\Phi'''_f(\rho)\right),
\end{align*}
where $\mathfrak{w}\in\mathbb{U}$ is a complex number. Considering the definition of $\rho(n)$ and~\eqref{uniquesoltuion}, the integrand in \eqref{beforetaylortheorem} can be expressed as
\begin{align*}
\rho^{-n}\int_{-\tau}^{\tau}\exp(\Lambda_f(\rho,\Theta))d\Theta,
\end{align*}

where 
\begin{align*}
\Lambda_f(\rho,\Theta)
=&\Phi_f(\rho)-\frac{1}{2}\Theta^24\pi^2(\rho\Phi'_f(\rho)+\rho^2\Phi''_f(\rho))\\
&+\frac{1}{3}\mathfrak{w}|\Theta|^3(8\pi^3\rho\Phi'_f(\rho)+24\pi^3\rho^2\Phi''_f(\rho)+8\pi^3\rho^3\Phi'''_f(\rho)).
\end{align*}

Adopting the argument presented in \cite[\S4]{vaughan} and Lemma~\ref{majorarcsestimate}, we have
\begin{align*}
&\rho\Phi'_f(\rho)+\rho^2\Phi''_f(\rho)
\gg X^3c_f(\log\log X+\psi_f),\\
&8\pi^3\rho\Phi'_f(\rho) +24\pi^3\rho^2\Phi''_f(\rho)+8\pi^3\rho^3\Phi'''_f(\rho)
\ll X^4c_f(\log\log X+\psi_f).
\end{align*}

Thus, for $|\Theta|\leq \tau$, there exists a constant $C_2$ for sufficiently large $X$, yielding
\begin{align*}
&\left|\frac{1}{3}\mathfrak{w}|\Theta|^3(8\pi^3\rho\Phi'_f(\rho)+24\pi^3\rho^2\Phi''_f(\rho)+8\pi^3\rho^3\Phi'''_f(\rho))\right|\\ \leq&C_2\Theta^2X^3c_f(\log\log X+\psi_f)(\log X)^{-1/4}
\leq C_2\Theta^2X^3c_f(\log\log X+\psi_f)\\
\leq&\pi^2\Theta^2(\rho\Phi'_f(\rho)+\rho^2\Phi''_f(\rho)).
\end{align*}

Hence,
\begin{align*}
\mathrm{Re}(\Lambda_f(\rho,\Theta))
\leq\Phi_f(\rho)-\pi^2\Theta^2(\rho\Phi'_f(\rho)+\rho^2\Phi''_f(\rho)).
\end{align*}

For $|\Theta|\geq X^{-3/2}(\log\log X+\psi_f)^{-1}$, there exist a positive constant $C_3$ such that
\begin{align*}
\mathrm{Re}(\Lambda_f(\rho,\Theta))
\leq\Phi_f(\rho)-C_3(\log\log X+\psi_f)
\end{align*}

Therefore, the contribution of the integrand in \eqref{beforetaylortheorem} yields
\begin{align}\label{seconderrorterm}
\int_{X^{-3/2}(\log\log X)^{-1}\leq|\Theta|\leq\tau}\exp(\Phi_f(\rho e(\Theta)))e(-n\Theta)d\Theta
\ll \Psi_f(\rho)X^{-C_3(\log\log X+\psi_f)}
\ll\Psi_f(\rho)n^{-B_2}.
\end{align}

Now it remains to deal with the integrand for the interval
\begin{align*}
\left[-X^{-3/2}(\log\log X+\psi_f)^{-1},X^{-3/2}(\log\log X+\psi_f)^{-1}\right].    
\end{align*}

For $\Theta$ belonging to the above interval we have
\begin{align*}
&|\Theta|^3\left(8\pi^3\rho\Phi'_f(\rho)+12\pi^3\rho^2\Phi''_f(\rho)+8\pi^3\rho^3\Phi'''_f(\rho)\right)\\
\ll&X^{-\frac{9}{2}}(\log\log X)^{-3}X^4\log\log X
=X^{-\frac{1}{2}}(\log\log X)^{-2}.
\end{align*}

From Lemma~\ref{asymptoticestimate}; recall that $n=x\asymp X^2(\log\log X+\psi_f)$. Thus
\begin{align*}
X^{\frac{1}{2}}(\log\log X)^{-2}
=&\left(X^2(\log\log X+\psi_f)\right)^{\frac{1}{2}}(\log\log X)^{-2}
\gg n(\log\log X+\psi_f)^{\frac{1}{2}}
\gg n^{1-\varepsilon}.
\end{align*}

Then
\begin{align*}
|\Theta|^3\left(8\pi^3\rho\Phi'_f(\rho)+12\pi^3\rho^2\Phi''_f(\rho)+8\pi^3\rho^3\Phi'''_f(\rho)\right)
\ll n^{-1+\varepsilon},
\end{align*}

and
\begin{align*}
\exp\left(|\Theta|^3\left(8\pi^3\rho\Phi'_f(\rho)+12\pi^3\rho^2\Phi''_f(\rho)+8\pi^3\rho^3\Phi'''_f(\rho)\right)\right)
=1+O\left(n^{-1+\varepsilon}\right).
\end{align*}

By definition $\Phi_{f,(2)}(\rho)=\rho\Phi'_f(\rho)+\rho^2\Phi''_f(\rho)$. Then we have

\begin{align*}
&\int_{-X^{-3/2}(\log\log X+\psi_f)^{-1}}^{X^{-3/2}(\log\log X+\psi_f)^{-1}}\exp(\Phi_f(\rho e(\Theta)))e(-n\Theta)d\Theta\\
=&\int_{-X^{-3/2}(\log\log X+\psi_f)^{-1}}^{X^{-3/2}(\log\log X+\psi_f)^{-1}}\exp(\Lambda_f(\rho,\Theta))d\Theta\\
=&(1+O(n^{-1+\varepsilon}))\Psi_f(\rho)\int_{-X^{-3/2}(\log\log X+\psi_f)^{-1}}^{X^{-3/2}(\log\log X+\psi_f)^{-1}}\exp(-2\pi\Theta^2\Phi_{f,(2)}(\rho))d\Theta.
\end{align*}

Recall that $X^{-3/2}(\log\log X+\psi_f)^{-1}\Phi_{f,(2)}(\rho)\gg X^{\varepsilon}$. By performing a standard polar coordinates integration, we obtain
\begin{align*}
&\left(\int_{-X^{-3/2}(\log\log X+\psi_f)^{-1}}^{X^{-3/2}(\log\log X+\psi_f)^{-1}}\exp\left(-\Theta^22\pi^2\Phi_{f,(2)}(\rho)d\Theta\right)\right)^2\\
=&\frac{1}{2\pi\Phi_{f,(2)}(\rho)}\left(1-\exp(-X^{-3}(\log\log X+\psi_f))^{-2}2\pi^2\Phi_{f,(2)}(\rho))\right)\\
=&\frac{1}{2\pi\Phi_{f,(2)}(\rho)}\left(1+O\left(\frac{1}{\log X}\right)\right).
\end{align*}

Therefore,
\begin{align*}
&\int_{-X^{-3/2}(\log\log X+\psi_f)^{-1}}^{X^{-3/2}(\log\log X+\psi_f)^{-1}}\exp(\Phi_f(\rho e(\Theta)))e(-n\Theta)d\Theta\\
=&\frac{\Psi_f(\rho)}{\sqrt{2\pi\Phi_{f,(2)}(\rho)}}\left(1+O\left(n^{-1+\varepsilon}\right)\right)\left(1+O\left(\frac{1}{\log X}\right)\right).
\end{align*}
Combining the aforementioned expression with~\eqref{beforetaylortheorem} and~\eqref{seconderrorterm} leads to the completion of the proof of the theorem.
\end{proof}

\begin{remark}
In Theorem~\ref{secondarytheorem}, it is possible to obtain a more accurate error term by obtaining a more precise estimate for the constant term $\psi_f$ as defined in~\eqref{constants}. However, achieving this would require making additional assumptions about $f$. To avoid imposing stricter conditions, we have opted for a slightly weaker but still acceptable error term, which applies to all positive real valued additive functions $f\in\mathcal{A}$.
\end{remark}

\subsection{Proof of Theorem~\ref{maintheorem}}
We start by recalling the following relation from Lemma~\ref{asymptoticestimate},
\begin{align*}
\rho^{-n}\Psi_f(\rho)
=&\exp\left(n\log\frac{1}{\rho(n)}+\Phi_f(\rho(n))\right)\\
=&\exp\left(\left(n\zeta(2)c_f(\log\log n+\psi_f)\right)^{\frac{1}{2}}(1+o(1))\right).
\end{align*}

Moreover,
\begin{align*}
\sqrt{2\pi\Phi_{f,(2)}(\rho(n))}
=\sqrt{2\pi\Gamma(3)}n^{\frac{3}{4}}\left(\frac{1}{\zeta(2)c_f(\log\log n+\psi_f)}\right)^{\frac{1}{4}}(1+o(1)).
\end{align*}
Combining the above expressions and applying Theorem~\ref{secondarytheorem} concludes the proof of the main result.

\qed

\subsection{The difference function}\label{gapfunction}
Let $\rho=\rho(n)$. Then, we have
\begin{align*}
p_f(n+1)-p_f(n)
=\rho^{-n}\int_{-\frac{1}{2}}^{\frac{1}{2}}\Psi_f(\rho e(\Theta))(\rho^{-1}e(-\Theta)-1)e(-n\Theta)d\Theta.
\end{align*}

The cases when $|\Theta|>X^{-3/2}(\log\log X+\psi_f)^{-1}$, we handle them similarly to the proof of Theorem~\ref{secondarytheorem}, where the contribution of the integrand is bounded by
\begin{align*}
\ll \rho^{-n}\Psi_f(\rho)n^{-B_2}.    
\end{align*}

In the case of $|\Theta|\leq X^{-3/2}(\log\log X+\psi_f)^{-1}$, we have
\begin{align*}
\rho^{-1}e(-\Theta)-1
=\frac{1}{X}+O(X^{-3/2}(\log\log X+\psi_f)^{-1})
=\frac{1}{X}(1+O(n^{-1+\varepsilon})).
\end{align*}

Thus, as given in the proof of Theorem~\ref{secondarytheorem}
\begin{align*}
\Psi_f(\rho e(\Theta))(\rho^{-1}e(-\Theta)-1)e(-n\Theta)
=(1+O(n^{-1+\varepsilon}))\Psi_f(\rho)\exp(-\Theta^22\pi^2\Phi_{f,(2)}(\rho)).
\end{align*}

Then, by Theorem~\ref{secondarytheorem} and~\eqref{bigxvalue}, we get
\begin{align*}
p_f(n+1)-p_f(n)
=\frac{\rho^{-n}\log\left(\frac{1}{\rho}\right)\Psi_f(\rho)}{\sqrt{2\pi\Phi_{f,(2)}(\rho)}}(1+O(n^{-1+\varepsilon})).
\end{align*}

\section{Example with Prime-omega function}\label{primeomegaexample}

While employing a different function $f\in\mathcal{A}$ in the generating series~\eqref{generatingfunction}, one arrives at a distinct partition problem each time. In this section, we illustrate an example of Theorem~\ref{maintheorem} using the prime-omega function $\omega(n)$, which represents the number of distinct prime factors of the natural number $n$ and is a strongly additive function\footnote{The prime-omega function holds significance in analytic number theory, as discussed in the introduction of~\cite{gransound}.}. In this scenario, we express the number of ways to write positive integer $n$ as follows
\begin{align*}
n=n_1\omega(n_1)+n_2\omega(n_2)+\cdots+n_s\omega(n_s),    
\end{align*}

where $\{n_j\}_{j=1}^{s}$ is a sequence of increasing positive integers. Then we can interpret the problem with the generating series
\begin{align}\label{omegageneratingfunction}
\Psi_{\omega}(z)
=\sum_{n\geq 0} p_{\omega}(n)z^n
=\prod_{n\in\mathbb{N}^{*}}(1-z^n)^{-\omega(n)}
=\exp(\Phi_{\omega}(z)),\quad (z\in\mathbb{U})
\end{align}
with
\begin{align}\label{omegaloggeneratingfunction}
\Phi_{\omega}(z)=\sum_{k=1}^{\infty}\sum_{n\in\mathbb{N}^{*}}\frac{\omega(n)}{k}z^{nk}.    
\end{align}

Our next result is stated as follows.

\begin{theorem}\label{primeomegatheorem}
Let $p_{\omega}(n)$ denotes the weighted partition into prime-omega function. Then as $n\to\infty$
\begin{align*}
p_{\omega}(n)
\sim c_1 n^{-\frac{3}{4}}(\log\log n+M)^{\frac{1}{4}} \exp\left(c_2\left(n(\log\log n+M)\right)^{\frac{1}{2}}(1+o(1))\right),
\end{align*}
where
\begin{align*}
c_1=\frac{\zeta(2)^{\frac{1}{4}}}{\sqrt{4\pi}},\quad c_2=\zeta(2)^{\frac{1}{2}}, \quad\text{and}\quad M=\gamma+\sum_{p\in\mathbb{P}}\left(\log\left(1-\frac{1}{p}\right)+\frac{1}{p}\right)
\end{align*}
is the Meissel-Mertens constant.
\end{theorem}

The proof directly follows from the argument of Theorem~\ref{maintheorem}, and thus the setup of the arcs in~\eqref{omegageneratingfunction} proceeds similarly. \footnotetext{The constant $c_{\omega}=1$ implies $\psi_{\omega}=\gamma+\mathcal{D}(1)=M$, as defined in~\eqref{constants}. Here, $\mathcal{D}(1)\approx-0.3157\ldots$ represents Fr\"{o}berg's constant~\cite{froberg}.}

\subsection{Contribution of the arcs}

For $\mathrm{Re}(s)>1$
\begin{align*}
\zeta_{\omega}(s)
=\sum_{n=1}^{\infty}\frac{\omega(n)}{n^s}
=\zeta(s)\zeta_{\mathbb{P}}(s),
\end{align*}
where $\zeta_{\mathbb{P}}$ denotes the prime-zeta function given by
\begin{align*}
\zeta_{\mathbb{P}}(s)
=\sum_{p\in\mathbb{P}}\frac{1}{p^s}
=\log\zeta(s)-\mathcal{D}(s)
\end{align*}
with $\mathcal{D}(s)=\sum_{k=2}^{\infty}\frac{1}{k}\sum_{p}\frac{1}{p^{ks}}$. For any $\delta_1>0$ we have that $\mathcal{D}(s)$ converges absolutely and uniformly for $\mathrm{Re}(s)\geq\frac{1}{2}+\delta_1$. Therefore, we have the follwoing relation
\begin{align*}
\zeta_{\mathbb{P}}(s)=\log\zeta(s)-\mathcal{D}(s).
\end{align*}

Note that the fundamental estimate derived in Lemma~\ref{majorarcsestimate} can be directly applied for $\omega(n)$, leading to the following lemma.

\begin{lemma}\label{fundamentalforomega}
Suppose that $\rho=e^{-\frac{1}{X}}$. Then as $X\to\infty$, one has that
\begin{align*}
\Phi_{\omega,(m)}(\rho)
=X^{m+1}\zeta(2)\Gamma(m+1)(\log\log X+M)\left(1+O\left(\frac{1}{\log X}\right)\right),
\end{align*}
and
\begin{align*}
\Phi_{\omega}^{(m)}(\rho)
=X^{m+1}\zeta(2)\Gamma(m+1)(\log\log X+M)\left(1+O\left(\frac{1}{\log X}\right)\right).
\end{align*}
\end{lemma}

An interesting consequence of Merten's Theorem on the distribution of $\omega(n)$ for $q\leq (\log N)^{A}$ and $(\ell,q)=1$ is given by
\begin{align}\label{omegaoverarithmeticprogression}
\omega_{\ell,q}(N)
=\sum_{\substack{n\leq N\\n\equiv\ell(\bmod{q})}}\omega(n)
=&\frac{1}{\varphi(q)}\sum_{\substack{n\leq N\\(n,q)=1}}\omega(n)+\mathcal{C}N+O\left(\frac{N}{\log N}\right)\nonumber\\
=&\frac{N}{\varphi(q)}\log\log N+\mathcal{C}N+O\left(\frac{N}{\log N}\right),
\end{align}

where $\mathcal{C}$ is a constant.

\begin{lemma}\label{meanvalueforomega}
With the same choice of parameters as in Lemma~\ref{meanvaluelemma}, one has that
\begin{align*}
\left|\Phi_{\omega}(\rho e(\Theta))\right|
\leq \frac{1}{|\tau|q^2}\zeta(2)(-1)^{\omega(q)}X\log\log X\prod_{p|q}p\left(1+O\left(\frac{1}{\log X}\right)\right).
\end{align*}
\end{lemma}

\begin{proof}
Following the argument of the proof of Lemma~\ref{meanvaluelemma}, applying Cauchy's inequality and~\eqref{omegaloggeneratingfunction} yields
\begin{align*}
\Phi_{\omega}(\rho e(\Theta))
=\Phi_{\omega}\left(\rho e\left(\beta+\frac{a}{q}\right)\right)
=\sum_{k=1}^{N}\frac{1}{k}\left(\sum_{\substack{\ell=1\\(\ell,q_k)=1}}^{q_k}e\left(\frac{a_k\ell}{q_k}\right)\sum_{n\equiv\ell(\bmod{q_k})}\omega(n)\exp(-kn\tau/X)+O(q_k^{\varepsilon})\right)\\
\quad\quad\quad\quad\quad\quad\quad\quad\quad\quad\quad\quad\quad\quad\quad\quad\quad\quad\quad\quad\quad\quad\quad\quad\quad\quad\quad\quad+O\left(\frac{X}{N}\right).
\end{align*}

Since the parameter $N$ is at our disposal, choosing $N=\sqrt{X}$, by Abel summation formula and~\eqref{omegaoverarithmeticprogression} the inner sum of the above expression becomes
\begin{align}\label{beforeh}
\sum_{n\equiv\ell(\bmod{q_k})}\omega(n)\exp(-kn\tau/X)
=\frac{k\tau}{X}\int_{2}^{\infty}\left(\frac{u}{\varphi(q_k)}\log\log u+u\mathcal{C}+O\left(\frac{u}{\log u}\right)\right)\exp(-ku\tau/X) du.
\end{align}

Substituting $h(u)
=u\mathcal{C}+O\left(\frac{u}{\log u}\right)$ into~\eqref{exponentialidentity}, we obtain
\begin{align}\label{hprimevalue}
h'(u)
=&\mathcal{C}+O\left(\frac{1}{\log u}\right).
\end{align}

Focusing on the constant and error terms of~\eqref{hprimevalue} and applying it to the integrand~\eqref{beforeh}, we obtain the bound
\begin{align}\label{constantvaluebound}
&\frac{k\tau}{X}\int_{2}^{X}\left(u\mathcal{C}+O\left(\frac{u}{\log u}\right)\right)\exp(-ku\tau/X) du\nonumber\\
=&\int_{2}^{\infty}\left(\mathcal{C}+O\left(\frac{1}{\log u}\right)\right)\exp(-ku\tau/X) du
\ll\frac{X}{k\tau}\left(1+\frac{1}{\log X}\right).
\end{align}

In the above expression, we have utilized the fact that the integrand is always less than 1 for all $u$, following a similar approach as described in Lemma~\ref{meanvaluelemma}.

\medskip


Let 
\begin{align*}
&t=\log\log u,\quad s=-u\exp(-k\tau u/X),\\
& dt=\frac{du}{u\log u},\quad\text{and}\quad 
ds=\frac{k\tau}{X}u\exp(-k\tau u/X)du
\end{align*}

Applying integration by parts yields 
\begin{align}\label{beforecut}
\frac{k\tau}{X\varphi(q_k)}\int_{2}^{\infty}u\log\log u\exp\left(-{k\tau u}/{X}\right) du
\leq &\frac{1}{\varphi(q_k)}\int_{2}^{\infty}\frac{\exp(-k\tau u/X)}{\log u}du.
\end{align}

For the integral in~\eqref{beforecut} we can directly employ the argument of Lemma~\ref{meanvaluelemma} to conclude the proof.
\end{proof}


The contribution of the minor arcs readily follows from Lemma~\ref{minorarclemma} as an application of Corollary~\ref{expsumwithadditivecoefficient}.

\subsection{Proof of Theorem~\ref{primeomegatheorem} using the saddle-point method}

From Lemma~\ref{asymptoticestimate}, for $\rho=\rho(x)$, as $x\to\infty$ we have
\begin{align}\label{lemma51foromega}
& x\log\frac{1}{\rho(x)}
=(x\zeta(2)(\log\log x+M))^{\frac{1}{2}}\left(1+O\left(\frac{1}{\log x}\right)\right),\quad\text{and}\nonumber\\
&\Phi_{\omega,(m)}(\rho(x))
=\Gamma(m+1)x^{\frac{m+1}{2}}(\zeta(2)(\log\log x+M))^{\frac{1-m}{2}}\left(1+O\left(\frac{1}{\log x}\right)\right).
\end{align}

The result below follows from Theorem~\ref{secondarytheorem}, but by substituting $\omega(n)$ in place of any $f\in\mathcal{A}$, we attain an improved error term. We give a brief description of the proof.

\begin{corollary}\label{secondarycorollary}
Let $\rho=\rho(n)$, then
\begin{align*}
p_{\omega}(n)
=\frac{\rho^{-n}\Psi_{\omega}(\rho)}{\sqrt{2\pi\Phi_{\omega,(2)}(\rho)}}\left(1+O(n^{-\frac{1}{5}})\right).
\end{align*}
\end{corollary}

\begin{proof}
Using the same parameters as in Theorem~\ref{secondarytheorem}, along with Lemmas~\ref{fundamentalforomega} and~\ref{meanvalueforomega}, we write
\begin{align*}
\mathrm{Re}(\Phi_{\omega}(\rho e(\Theta)))
\leq (1-(\log X)^{-1})\Phi_{\omega}(\rho).
\end{align*}

Thus by Lemma~\ref{fundamentalforomega}
\begin{align*}
\Phi_{\omega}(\rho e(\Theta))
\ll \Phi_{\omega}(\rho)n^{-10}.
\end{align*}

Hence by~\eqref{cauchyformula}
\begin{align}\label{beforetaylorforomega}
p_{\omega}(n)
=\rho^{-n}\int_{-\tau}^{\tau}\exp(\Phi_{\omega}(\rho e(\Theta)))e(-n\Theta) d\Theta+O\left(\frac{\Psi_{\omega}(\rho)}{\rho^nn^{10}}\right).
\end{align}

Now we follow the argument of Theorem~\ref{secondarytheorem}, and by Lemma~\ref{fundamentalforomega}, we have
\begin{align*}
&\rho\Phi_{\omega}'(\rho)+\rho^2\Phi_{\omega}''(\rho)
\gg X^3(\log\log X+M),\\
&8\pi^3\rho\Phi_{\omega}'(\rho)+24\pi^3\rho^2\Phi_{\omega}''(\rho)+8\pi^3\rho^3\Phi_{\omega}'''(\rho)
\ll X^4(\log\log X+M).
\end{align*}

The case where $|\Theta|\leq \tau$ directly follows from Theorem~\ref{secondarytheorem}. Similarly, for $|\Theta|\geq X^{-\frac{3}{2}}(\log\log X)^{-1}$, we bound the contribution of the integrand in~\eqref{beforetaylorforomega} by
\begin{align}\label{aftertaylorforomega}
\int_{X^{-3/2}(\log\log X)^{-1}\leq |\Theta|\leq \tau}\exp(\Phi_{\omega}(\rho e(\Theta)))e(-n\Theta) d\Theta
\ll \Psi_{\omega}(\rho)n^{-10}.
\end{align}

Now, let us consider the integral in~\eqref{beforetaylorforomega} for $\Theta\in\left[-X^{-3/2}(\log\log X)^{-1},X^{-3/2}(\log\log X)^{-1}\right]$. Recall that $n=x\asymp X^2\log\log X$ from Lemma~\ref{fundamentalforomega}, we have
\begin{align*}
&|\Theta|^3\left(8\pi^3\rho\Phi'_{\omega}(\rho)+12\pi^3\rho^2\Phi''_{\omega}(\rho)+8\pi^3\rho^3\Phi'''_{\omega}(\rho)\right)\\
\ll& X^{-\frac{9}{2}}X^4(\log\log X)^{-3}(\log\log X)
\ll X^{-\frac{1}{2}}(\log\log X)^{-2}\\
\ll& n^{-\frac{1}{4}}(\log\log X)^{-\frac{1}{4}}(\log\log X)^{-2}
\ll n^{-\frac{1}{5}}.
\end{align*}

By applying the argument from Theorem~\ref{secondarytheorem}, we conclude the proof.
\end{proof}

\medskip

{\it Proof of Theorem~\ref{primeomegatheorem}.} From~\eqref{lemma51foromega} and Corollary~\ref{secondarycorollary}, we obtain
\begin{align*}
\rho^{-n}\Psi_{\omega}(\rho)
=&\exp\left(n\log\frac{1}{\rho(n)}+\Phi_{\omega}(\rho(n))\right)\\
=&\exp\left(\left(\zeta(2)n(\log\log n+M)\right)^{\frac{1}{2}}(1+o(1))\right).
\end{align*}

Additionally, we have
\begin{align*}
\sqrt{2\pi\Phi_{\omega,(2)}(\rho(n))}
=(2\pi\Gamma(3))^{\frac{1}{2}}n^{\frac{3}{4}}({\zeta(2)(\log\log n+M)})^{-\frac{1}{4}}(1+o(1)),
\end{align*}
which completes the proof.
\qed






\subsection*{Acknowledgment} I would like to express my gratitude to Alexandru Zaharescu for suggesting this problem and to Kyle Pratt for their valuable suggestions and insightful discussions. Special thanks to Nicolas Robles for his encouragement and help with Lemma~\ref{majorarcsestimate}. I am grateful to Nigel Byott for reviewing the draft and providing feedback. This project is funded by EPSRC doctoral fellowship \#EP/T518049/1.

\medskip

{\it{Rights Retention}}. For the purpose of open access, the author has applied a Creative Commons Attribution (CC BY) licence to any Author Accepted Manuscript version arising from this submission.

\printbibliography 

\end{document}